%% file: analysis-pp.tex
\documentclass[preprint]{imsart}
\usepackage{amsthm,amsmath}
\RequirePackage{natbib}
\RequirePackage[colorlinks,citecolor=blue,urlcolor=blue]{hyperref}
\usepackage{amssymb}
\usepackage{accents,xcolor}
\usepackage{mathtools}
\usepackage{mleftright}
\usepackage{lscape} 
\usepackage{xfrac,enumerate}
\usepackage[all]{xy}

\doi{10.1214/154957804100000000}
\pubyear{0000}
\volume{0}
\firstpage{0}
\lastpage{0}
\arxiv{0000.00000}

\startlocaldefs
\newcommand{\ex}{\mathbb{E}}

\newcommand{\pr}{\mathbb{P}}
\newcommand{\R}{\mathbb{R}}
\newcommand{\bigo}{\mathcal{O}}
\newcommand{\N}{\mathcal{N}}

\newcommand{\cals}{\mathcal{S}}

\DeclarePairedDelimiterX{\norm}[1]{\lVert}{\rVert}{#1}
\DeclarePairedDelimiterX{\abs}[1]{\lvert}{\rvert}{#1}
\DeclarePairedDelimiterX{\0norm}[1]{\lVert}{\rVert_{0}}{#1}
\DeclarePairedDelimiterX{\1norm}[1]{\lVert}{\rVert_{1}}{#1}
\DeclarePairedDelimiterX{\2norm}[1]{\lVert}{\rVert_{2}}{#1}
\DeclarePairedDelimiterX{\nnorm}[1]{\lVert}{\rVert_{n}}{#1}
\DeclarePairedDelimiterX{\2nnorm}[1]{\lVert}{\rVert_{n}^2}{#1}

\newcommand{\he}{\hat{\epsilon}}
\newcommand{\hf}{\hat{f}}
\newcommand{\hft}{\hat{f}_t}
\newcommand{\e}{\epsilon}
\newcommand{\s}{\sigma}
\newcommand{\lam}{\lambda}

\newcommand{\hs}{\hat{\s}}
\newcommand{\hb}{\hat{\beta}}

\newcommand{\het}{\hat{\epsilon}_t}
\newcommand{\nee}{\norm{\e}}
\newcommand{\neh}{\norm{\he}}
\newcommand{\net}{\norm{\het}}
\newcommand{\RR}{\mathcal{R}}
\newcommand{\A}{\mathcal{A}}

\let\oldsqrt\sqrt
\def\sqrt{\mathpalette\DHLhksqrt}
\def\DHLhksqrt#1#2{%
\setbox0=\hbox{$#1\oldsqrt{#2\,}$}\dimen0=\ht0
\advance\dimen0-0.2\ht0
\setbox2=\hbox{\vrule height\ht0 depth -\dimen0}%
{\box0\lower0.4pt\box2}}

\newtheoremstyle{break}
  {\topsep}{\topsep}%
  {\itshape}{}%
  {\bfseries}{}%
  {\newline}{}%
\theoremstyle{break}
\newtheorem{definition}{Definition}[section]
\newtheorem{lemma}{Lemma}[section]
\newtheorem{theorem}{Theorem}[section]
\newtheorem{proposition}{Propostion}[section]
\newtheorem{Coro}{Corollary}[section]
\newtheorem{assumption}{Assumption}[section]
\theoremstyle{definition}

\newtheorem*{note}{Note}
\newtheorem*{remark}{Remark}

\endlocaldefs

\usepackage{Sweave}
\begin{document}
\input{analysis-pp-concordance}
\begin{frontmatter}
\title{Oracle inequalities for square root analysis estimators with application to total variation penalties}
\runtitle{Oracle inequalities for square root  analysis estimators}


\author{\fnms{Francesco} \snm{Ortelli}\ead[label=e1]{francesco.ortelli@stat.math.ethz.ch}} \and  \author{\fnms{Sara} \snm{van de Geer}\corref{}\ead[label=e2]{geer@stat.math.ethz.ch}}
\address{R\"{a}mistrasse 101\\ 8092 Z\"{u}rich\\ \printead{e1}; \printead{e2}}

\affiliation{Seminar for Statistics, ETH Z\"{u}rich}

\runauthor{Ortelli, van de Geer}

\begin{abstract}
Through the direct study of the analysis estimator we derive oracle inequalities with fast and slow rates by adapting the arguments involving projections by \cite{dala17}. We then extend the theory to the square root analysis estimator. Finally, we focus on (square root) total variation regularized estimators on graphs and obtain constant-friendly rates, which, up to log-terms, match  previous results obtained by entropy calculations. We  also obtain an oracle inequality for the (square root) total variation regularized estimator over the cycle graph.
\end{abstract}



\end{frontmatter}
\tableofcontents

\section{Introduction}

\subsection{Review of the literature}

\subsubsection{Synthesis and analysis}
In the literature we find two approaches to regularized empirical risk minimization: the synthesis and the analysis approach, see \cite{elad07}. Given a dictionary $X \in \R^{n \times p}$, the synthesis approach to the estimation of $f^0\in \R^n$ is expressed by the {synthesis estimator}
$$ \hat{f}= X \hat{\beta}, \ \hat{\beta}= \arg \min_{\beta \in \R^p} \left\{\norm{Y-X \beta}^2_n + 2 \lambda \norm{\beta}_1 \right\},\ \lambda>0,$$
where $Y= f^0 + \epsilon,\ \epsilon \sim \N_n(0, \sigma^2 \text{I}_n),\ \sigma\in (0, \infty)$, and for a vector $f \in \R^n$ we write $\norm{f}^2_n= \sum_{i=1}^n f_i^2/n$. An  instance of synthesis estimator is the classical lasso (\cite{tibs96}, see \cite{buhl11} and \cite{vand16} for a thorough exposition of the theory about the lasso).

On the other side, for an analysis operator $D\in \R^{m \times n}$, the {analysis estimator} is given by
$$ \hat{f}= \arg \min_{f \in \R^n} \left\{ \norm{Y-f}^2_n + 2 \lambda \norm{Df}_1 \right\},\ \lambda>0.$$
The analysis approach to the estimation of $f^0$ has previously been studied in e.g. \cite{vait13} and \cite{nam13}.
Instances of analysis estimators are total variation regularized estimators over graphs, in particular the fused lasso (\cite{tibs05}), which corresponds to the case of the path graph. For such estimators, $D$ is taken to be the incidence matrix of some directed graph $\vec{G}=(V,E)$.

Algorithms to solve both the analysis and the synthesis problem are exposed in \cite{tibs11}.

\subsubsection{Total variation regularized estimators}\label{analysis-sss112}

Let $\vec{G}=(V,E)$ be a general directed graph, where the set $V=[n]$ is the set of vertices and the set $E=\{e_1, \ldots, e_m\}$ is the set of edges. Every edge $e_i=(e_i^-, e_i^+)$ is directed from a vertex $e_i^-\in V$ to a vertex $e_i^+\in V$, $e_i^-\not= e_i^+$.

We define $D_{\vec{G}} \in \{-1,0,1\}^{m \times n}$, the the incidence matrix of $\vec{G}$, as
$$ (d_i')_j= \begin{cases} -1, &j=e_i^-,\\
1, & j=e_i^+,\\
0, & \text{else}, \end{cases}$$

where $d_i', i \in [m]$ denote the $i^{\text{th}}$ row of $D_{\vec{G}}$.
Total variation regularized estimators are analysis estimators, where the anaylsis operator $D$ is taken to be $D=D_{\vec{G}}$ for some graph $\vec{G}$. Thus, the differences of the candidate estimator $f$ across the edges of the graph $\vec{G}$ are penalized.

Some previous studies of total variation regularized estimators  (\cite{dala17,orte18}) used a {step through a synthesis formulation} (cf. \cite{orte19-1}) to prove oracle inequalities. However, these studies were confined to restrictive graph structures: the path in \cite{dala17} and a class of tree graphs in \cite{orte18}. Other studies focusing on the fused lasso and not directly involving its synthesis form also implicitly relied on some kind of dictionary to handle the error term by projections onto some columns of this dictionary, see for instance the lower interpolant by \cite{lin17b}.

The approach by \cite{hutt16}, in spite of handling directly the analysis estimator, is not able to  guarantee the convergence of the mean squared error for the fused lasso.

For $C>0$, define
$$ \mathcal{G}(C):= \left\{f\in \text{rowspan}(D): \norm{D_{\vec{G}}f}_1\le C  \right\}.$$

The minimax rate of estimation over the path graph for functions $f \in \mathcal{G}(C),\ C \asymp 1$ is $n^{-2/3}$ (\cite{dono98}).
Moreover the fused lasso tuned with $\lambda \asymp n^{-2/3} C^{-1/3}$  has $ \norm{\hat{f}-f^0}^2_n= \bigo_{\pr}(n^{-2/3}C^{2/3})$ if $f^0\in \mathcal{G}(C)$ and thus achieves the minimax rate (\cite{mamm97-2}).
This result is based on {entropy bounds} (see \cite{babe79, birm67}) on the class $\mathcal{G}(C)$, which are not constant-friendly.
On the opposite side, \cite{sadh16} showed that estimators given by linear transformations of the observations are suboptimal on $\mathcal{G}(C)$.

In a recent paper, \cite{padi17} prove that when $\vec{G}$ is a tree graph with bounded maximal degree and $f^0 \in \mathcal{G}(C)$, then the minimax rate is $n^{-2/3} C^{2/3}$.

Moreover,  \cite{padi17} prove that the total variation regularized estimator over any connected graph has a mean squared error of order at most $n^{-2/3}C^{2/3}$ if $f^0 \in \mathcal{G}(C)$. Thus the total variation regularized estimator over tree graphs of bounded maximal degree is proved to be minimax-optimal. This result is based on entropy bounds by \cite{wang16} and is not constant-friendly.

In \cite{sadh16}, the authors prove that, for $\vec{G}$ being the two dimensional grid graph, the minimax rate of estimation for  $f^0 \in \mathcal{G}(C)$ with the canonical scaling $C \asymp n^{1/2}$ is $\sqrt{{\log n}/{n}}$.  The paper by \cite{hutt16} shows that this rate is retrieved by the total variation regularized estimator up to log terms.

In a recent work, \cite{chat19} obtain convergence rates for the total variation regularized estimator over the two dimensional grid by proof techniques involving bounds on the Gaussian width of tangent cones.

These previous results will serve as a {benchmark} for the evaluation of the rates of the  oracle inequalities presented in this paper.

\subsubsection{Square root regularization}
The square root lasso estimator, defined as
$$\hat{\beta}:= \arg \min_{\beta\in \R^n} \left\{ \norm{Y-X\beta}_n + \lambda_0 \norm{\beta}_1 \right\},\ \lambda_0>0,$$
was first introduced by \cite{bell11} and  allows to simulataneously estimate the regression coefficients and the noise level. Thus,  when tuning the estimator to obtain oracle properties, one can {choose $\lambda_0$ not depending on the unknown noise level $\sigma$}. The square root lasso estimator is studied in \cite{bell11}, \cite{sun12} (where it is called scaled lasso), \cite{vand16} and \cite{stuc17}, among the others.

One can rewrite the minimization problem in the following form
$$ (\hat{\beta}, \hat{\sigma}):= \arg \min_{\beta \in \R^p, \sigma>0} \left\{ \frac{\norm{Y-X\beta}^2_n}{\sigma} + \s + 2 \lam_0\norm{\beta}_1 \right\},\ \lambda_0 >0.$$
The objective function of this second expression of the estimator is not differentiable at $\s=0$ and thus {if $\hs=0$ the KKT conditions do not hold}.
By differentiating the penalized loss and assuming that $\hs\not=0$ we get the KKT conditions
$$ \hat{\s}^2= \norm{Y-X \hat{\beta}}^2_n \text{ and } \frac{X'(Y-X\hb)}{n}= \lam_0 \hs \partial \norm{\hb}_1,$$
where $\partial \norm{\hb}_1$ is the subdifferential of $\norm{{\beta}}_1$ at $\beta=\hb$.

The papers \cite{bell11,sun12} propose algorithms to compute the square root lasso estimator, which are extended by \cite{bune14} and \cite{deru18} to the cases of the group square root lasso and of the square root slope respectively.

In this paper  we focus on analysis estimators. Our interest is motivated by the possibility to apply the results to the case of total variation regularization. As it will turn out in Theorems \ref{analysis-t01s03} and \ref{analysis-t02s03}, the choice of the tuning parameter $\lambda$ needed to ensure oracle properties for plain analysis estimators depends on the noise variance $\sigma^2$, which might be unknown. Therefore, we are interested in the square root version of the analysis estimator: the square root analysis estimator
$$ \hat{f}_{\sqrt{}}:= \arg \min_{f \in \R^n} \left\{\norm{Y-f}_n + \lam_0 \norm{Df}_1 \right\},\ \lam_0>0.$$
Indeed, square root estimators are known to be able to estimate the signal and the noise variance simultaneously and therefore allow for a choice of the tuning parameter $\lambda_0$ that does not depend on $\sigma$ to guarantee oracle inequalities. This will turn our to be the case in Theorem \ref{analysis-t01s04} and \ref{analysis-t02s04}.
Square root analysis estimators could be computed either by transforming them into square root synthesis estimators by using the insights provided by \cite{orte19-1} (which are largely based on \cite{elad07}) or by adapting to the square root case the algorithm provided by \cite{tibs11} to solve plain analysis problems.

We want to combine the arguments exposed by \cite{vand16} and \cite{dala17} and extend them to the square root analysis estimator.

\subsection{Contributions}
The main points profiling our results are:
\begin{itemize}
\item we study directly the analysis estimator without passing through its synthesis formulation;
\item we apply the projection arguments by \cite{dala17} to the case of square root regularization;
\item to do so we use  projection theory for analysis operators.
\end{itemize}
We make the following contributions:
\begin{enumerate}
\item We present a framework for proving oracle inequalities with fast and slow rates for a general analysis estimator without transforming the analysis estimation problem into a synthesis estimation problem. This constitutes an analysis counterpart of the results obtained by \cite{dala17} for the synthesis estimator.
\item We  introduce, inspired by some remarks by \cite{padi17}, $r_{S_0}:= \text{dim}(\N(D_{-S_0}))$ as measure for the sparsity of the signal (see Subsection \ref{analysis-s01ss03} for the notation). In \cite{hutt16}, the sparsity of the true signal was measured as $\norm{Df^0}_0$, while we argue that $r_{S_0}$ is  more appropriate.
\item For the total variation regularized estimator on the path graph, we show that an analogue of the bound on the increments of the empirical process by projections exposed by \cite{dala17}  is only off  by  log-terms from the one which can be obtained by entropy calculations, if we allow the tuning parameter $\lambda$ to depend on some aspects of $f^0\in \mathcal{G}(C)$. We thus match, up to log-terms, the result obtained   by means of entropy calculations by \cite{padi17} for general graphs  and by \cite{mamm97-2} for the path graph.
Note that entropy calculations are not constant-friendly, while the bounds we expose are and might be advantageous for a small enough value of $n$.
\item For the total variation regularized estimator over the cycle graph, we prove an oracle inequality with fast rates, which to our knowledge is a new contribution.
\item We adapt a lemma by \cite{vand16}, showing that the square root lasso does not overfit, to the case where the increments of the empirical process for the square root analysis estimator are bounded by means of the projection arguments by \cite{dala17}. This is a starting point for the development of {oracle inequalities for the square root analysis estimator}, which produce results analogous to the ones obtained for the plain analysis estimator (which match the ones found in \cite{dala17}). We then narrow down these results to square root total variation regularized estimators on graphs.
\end{enumerate}

\subsection{Notation}\label{analysis-s01ss03}
\textbf{Analysis operator $D$}. Let $D \in \R^{m \times n}$ be a given matrix. Let $\{ d'_i\}_{i \in [m]}$ denote the row vectors of $D$. By $\N(D)$, we denote the nullspace of $D$, i.e. $ \N(D):= \left\{ x \in \R^n: Dx=0 \right\}$.
Let $ \N^\perp (D):= \left\{ x \in \R^n: x'z=0, \forall z \in \N(D) \right\}$ denote the orthogonal complement of $\N(D)$. Note that $\N^\perp (D)= \text{rowspan}(D)$. By penalizing $\norm{Df}_1$, we favor an estimator lying almost in $\N(D)$, while we penalize estimators having high correlation with the rows of $D$.

\textbf{Active set $S \subset [m]$}. Let $S \subseteq [m]$ denote a subset of the row indices of $D$.
We denote the cardinality of the set $S$ by $s:= \abs{S}$. We write $-S:= [m] \setminus S$.
Moreover, we write $D_S= \{d_i'\}_{i \in S}\in \R^{s \times n}$ and $D_{-S}= \{d_i'\}_{i \in -S}\in \R^{(m-s) \times n}$. For instance, let us suppose that, for $S_0\subseteq[m]$, the true signal is s.t. $D_{S_0}f^0\not= 0$ and $D_{-S_0}f^0=0$. Then $S_0$ is the true active set for $Df^0$, i.e. the set of indices of rows of $D$, to which the true signal is not orthogonal. 

\textbf{Set of admissible active sets $\cals$}. Define $ S(f):=\text{support}(Df)= \{j\in [m]: d_j'f\not= 0\}$.
and 
$$ \cals=\cals(D):= \{S \subseteq [m] | \exists f \in \R^n: S=S(f)\}\subseteq \mathcal{P}([m]),$$
where $\mathcal{P}([m])$ denotes the power set of $[m]$. If $D$ is not of full row rank, then there might be some subsets of $[m]$, that can not be the active sets of $Df$ for any $f \in \R^n$. Thus, from now on, we restrict our attention to active sets $S \in \cals(D)$. More on this in Remark \ref{analysis-r01s05} in Section \ref{analysis-s05}.

\textbf{The nullspace $\N(D_{-S})$}.
 Note that, since $D_{-S(f)}f=0$, $f \in \N(D_{-S(f)})$. Thus $\N(D_{-S})$ encompasses all the signals $f$, s.t. $S\supseteq S(f)$. In a vector $f\in \R^n$ we have $n$ ``pieces'' of information. Note that $\N(D)$ can be nonempty and thus the part of $f$ lying in $\N(D)$ will always be  active, because it is not penalized. Moreover, since we can have $m>n$ and $\norm{Df}_0>n$, we see that $\norm{Df}_0=\abs{S(f)}$ is not a good measure for the sparsity of the signal. We thus use as a measure of sparsity $r_S=\text{dim}(\N(D_{-S}))  \le n$ to denote the pieces of information that the estimator effectively had to estimate if the active set were $S$.
 
We use the shorthand notations $\N_S:= \N(D_S)$ and $\N_{-S}:= \N(D_{-S})$. Similarly, we write $\N_S^\perp:= \N^\perp(D_S)$ and $\N_{-S}^\perp:= \N^\perp(D_{-S})$. Note that
$ \N(D)= \N(D_S) \cap \N(D_{-S}).$
Moreover, if $S,S'\subseteq [m]$ are s.t. $S \subset S'$, then we have that $ \N(D_{S})\supseteq \N(D_{S'})$.
In addition, if the rows of $D_{S'\setminus S}$ can be written as linear combinations of the rows of $D_S$, then $ \N(D_{S'})= \N(D_S)$.

\textbf{Diagonal matrices of weights}. Let $\tilde{w}\in \R^m$ be a vector, for instance a vector of weights. For the diagonal matrix $\tilde{W}=\text{diag}( \{\tilde{w}_i\}_{i \in [m]})\in \R^{m\times m}$ we write $\tilde{W}_S:=\text{diag}( \{\tilde{w}_i\}_{i \in S})\in \R^{s\times s}$ and $\tilde{W}_{-S}:=\text{diag}( \{\tilde{w}_i\}_{i \in -S})\in \R^{(m-s)\times (m-s)}$. We will need these notations for bounding the weighted weak compatibility constant, defined in Definition \ref{analysis-d04s01} below.

\textbf{Linear projections}. Let $\text{I}_n\in \R^{n \times n}$ denote the identity matrix and let $\mathbb{I}_n=\{1\}^{n \times n}$.

Let $V \subset \R^n$ be a linear space. By $\Pi_V \in \R^{n \times n}$ we denote the orthogonal projection matrix onto $V$ and by $A_V:= \text{I}_n- \Pi_V$ the orthogonal antiprojection matrix onto $V$.

Let $f \in \R^n$. We write $ f= (\Pi_{\N_{-S}}+ \Pi_{\N_{-S}^\perp})f=: f_{\N_{-S}}+ f_{\N_{-S}^\perp}$, i.e. for a set $S \in \cals(D)$ we decompose a signal $f$ into a low rank part (since usually $r_S$ will be small) orthogonal to $D_{-S}$ and a part collinear to $D_{-S}$.
We will use this decomposition when bounding the increments of the empirical processes in the proofs of the oracle inequalities.

Note that  $ \Pi_{\N_{-S}^\perp}= \text{I}_n - \Pi_{\N_{-S}}=: A_{\N_{-S}}$ and $ \Pi_{\N_{-S}}= \text{I}_n - \Pi_{\N_{-S}^\perp}=: A_{\N_{-S}^\perp}$.

\textbf{Computing $\Pi_{\N^{\perp}_{-S}}$}. 
 Let $S\in \cals$ be a set of row indices of $D$. We have that
$$ \Pi_{\N^\perp(D_{-S})}=\Pi_{\text{rowspan}(D_{-S})} = D_{-S}^+D_{-S},$$
where $D_{-S}^+ \in \R^{n \times (m-s)}$ denotes the Moore-Penrose pseudoinverse of $D_{-S}$.
If $D_{-S}\in \R^{(m-s)\times n}$ is of full row rank we have that $D_{-S}^+= D_{-S}'(D_{-S}D_{-S}')^{-1}$.

\subsection{Model assumptions and preliminary definitions}

\subsubsection{Model assumptions}
Throughout the paper we will use the following model, which assumes that we observe a signal contaminated with Gaussian noise. Let $f^0\in \R^n$ be a signal. We observe 
$$Y=f^0+ \epsilon,\ \epsilon \sim \N_n(0, \sigma^2 \text{I}_n),\ \sigma \in (0, \infty).$$ 

Moreover, for an analysis operator  $D \in \R^{m \times n}$ we will study the two following estimators.

\begin{itemize}
\item The \textbf{analysis estimator $\hat{f}$} of $f^0$, defined as 
$$ \hat{f}:= \arg \min_{f \in \R^n} \left\{ \norm{Y-f}^2_n+ 2 \lambda \norm{Df}_1 \right\},\ \lambda>0,$$

\item The \textbf{square root analysis estimator $\hat{f}_{\sqrt{}}$} of $f^0$, defined as
$$ \hat{f}_{\sqrt{}}:= \arg \min_{f \in \R^n} \left\{ \norm{Y-f}_n+ \lambda_0 \norm{Df}_1 \right\},\ \lambda_0>0.$$
\end{itemize}
In particular, Section \ref{analysis-s03} will focus on the study of the analysis estimator $\hat{f}$ for a general analysis operator $D$, while Section  \ref{analysis-s04} will deal with its square root counterpart $\hat{f}_{\sqrt{}}$. In Section \ref{analysis-s05} we will then apply the results of the two previous sections to total variation regularization.


\subsubsection{Definitions}

Let $D_{-S}^+$ denote the Moore-Penrose pseudoinverse of $D_{-S}$ and $d^+_i\in \R^n, i \in [m-s]$ the column vectors of $D^+_{-S}$.

We define the map $i^*: -S \mapsto [m-s]$, s.t. $i^*(i)= \sum_{j=1}^i 1_{\{j \in -S \}}$.
The index $i^*(i)$ denotes the  row index of the $i^{\text{th}}$ row of $D$, $i \in -S$, in the matrix $D_{-S}$.

We use a proof technique inspired by \cite{dala17}. The key aspect of this proof technique is to decompose the noise into two parts by using orthogonal projections:
\begin{itemize}
\item a part projected onto a low-rank linear subspace, which will be bounded by using the Cauchy-Schwarz inequality,
\item a remainder (i.e. the antiprojection), involving the weights defined below,  which will be bounded with more refined techniques.  These techniques involve, in the case of oracle inequalities with fast rates, the weak weighted compatibility constant. 
\end{itemize}

\begin{definition}[Weighted weak compatibility constant]\label{analysis-d04s01}
Let $\tilde{W} \in \R^{m \times m}$ be a diagonal matrix of weights with $\tilde{W}_S=\text{I}_s$ and $\norm{\tilde{W}}_\infty\le 1$ (e.g. as in Definition \ref{analysis-d03s01}). The weighted compatibility constant $\kappa^2(S, \tilde{W})$ is defined as
$$ \kappa^2(S,\tilde{W}):= r_S \min\left\{ \norm{f}^2_n: \norm{ D_S f}_1- \norm{\tilde{W}_{-S}D_{-S}f}_1=1 \right\}.$$

\end{definition}

\begin{remark}
This weighted compatibility constant extends the definition given by \cite{dala17} to the case of analysis estimators. A distinguishing feature is the factor $r_S$. When $S=S(f^0)=:S^0$, $r_{S_0}$ expresses the number of parameters to estimate . 

Note that the weak weighted compatibility constant relaxes the definition of compatibility constant given by \cite{hutt16}. There, one has to lower bound
$$ \frac{r_S \norm{f}^2_n}{\norm{D_Sf}_1^2},$$
while the weighted weak compatibility constant is applied to lower bound
$$ \frac{r_S \norm{f}^2_n}{(\norm{D_Sf}_1-\norm{\tilde{W}_{-S}D_{-S}f}_1)^2},$$
which is easier, since the denominator is smaller. The additional term $\norm{\tilde{W}_{-S}D_{-S}f}_1$ comes form the remainder term mentioned in the above sketch of the  proof technique used in this paper.

Note that bounds on the compatibility constant by \cite{hutt16} imply bounds on the weighted weak compatibility constant but the converse is not true. This is relevant, for instance, for the case of the total variation regularization over the path graph. In that case the bound by \cite{hutt16} is too rough. One can obtain more refined bounds by studying the weighted weak compatibility constant as done in \cite{dala17,orte18}. For instance, \cite{orte18} showed that the weighted weak compatibility constant also holds on a certain class of tree graphs. We will show in Section \ref{analysis-s05} a new bound on the weighted weak compatibility constant for the total variation regularized estimator over the cycle graph.
\end{remark}

\begin{definition}[Length of antiprojections]\label{analysis-d01s01}
In analogy to \cite{dala17}, the vector $ \omega\in \R^{m}$ is defined as
$$ \omega_i:= \begin{cases} \norm{d^+_{i^*(i)}}_n, & i\in -S,\\
0, & i \in S. \end{cases}$$
Moreover, we write $ \Omega:= \text{diag}( \{\omega_i\}_{i \in [m]})\in \R^{m\times m}$.
\end{definition}

\begin{note}
One can see that, if $D_{-S}$ is of full row rank,
$$ \omega_i^2= ((D_{-S}D_{-S}')^{-1}/n)_{i^*(i),i^*(i)},\ i \in -S.$$
\end{note}

We want to find a vector of weights with values in $[0,1]$, based on $\Omega$ defined above. We thus define hereafter the normalized scaling factor $\gamma$ as the maximum entry of $\Omega$.

\begin{definition}[Normalized inverse scaling factor]\label{analysis-d02s01}
In analogy to the quantity $\rho_T$ used by \cite{dala17} and to the scaling factor defined by \cite{hutt16}, the normalized inverse scaling factor $\gamma= \gamma(D, S(S))$ is defined as
$$ \gamma:= \max_{i \in -S} \omega_i .$$
\end{definition}
 We now normalize $\Omega$ by dividing its entries by $\gamma$ to obtain a vector of weights $w\in [0,1]^m$.
 
\begin{definition}[Weights]\label{analysis-d03s01}
In analogy to \cite{dala17}, the vector of weights $w \in \R^{m}$ is defined as
$$ w_i:= 1- \omega_i/\gamma,\ i \in [m].$$
Moreover, we write $ W:= \text{diag}( \{w_i\}_{i \in [m]})\in \R^{m\times m}$. Note that $W_S= \text{I}_s$.
\end{definition}

\section{Oracle inequalities for the analysis estimator}\label{analysis-s03}

In this section we study the analysis estimator, which is defined as
$$ \hat{f}:= \arg \min_{f \in \R^n} \left\{ \norm{Y-f}^2_n+ 2 \lambda \norm{Df}_1 \right\},\ \lambda>0.$$

This section produces analogous results to \cite{dala17}.  We however use an approach that does not take a detour via synthesis, but instead directly handles the analysis estimator. In Section \ref{analysis-s04} we are going to explore how this approach translates to the case of the square root analysis estimator.

\subsection{Fast rates with compatibility conditions}


\begin{theorem}[Oracle inequality with fast rates for the analysis estimator]\label{analysis-t01s03}
Let $S \in \mathcal{S}$ be arbitrary and $x, \ t >0$. Choose $\lambda \ge {\gamma\sigma }\sqrt{2\log(2(n-r_S))/n+ 2t/n} $. For the analysis estimator it holds that, $ \forall f \in \R^n$, with probability at least $1-e^{-x}-e^{-t}$,
\begin{equation*}
\norm{\hat{f}-f^0}^2_n  \le  \norm{f-f^0}^2_n + 4 \lambda \norm{D_{-S}f}_1 +  \left(   \sigma \sqrt{\frac{2x}{n}}+ \sigma \sqrt{\frac{r_S}{n}}+  \frac{\lambda \sqrt{r_S}}{\kappa(S,W)}\right)^2.
\end{equation*}
\end{theorem}


\begin{proof}[Proof of Theorem \ref{analysis-t01s03}]
See Appendix \ref{analysis-appC}.
\end{proof}


\subsection{Slow rates without compatibility conditions}

\begin{theorem}[Oracle inequality with slow rates for the analysis estimator]\label{analysis-t02s03}
Let $S \in \mathcal{S}$ be arbitrary and $x, \ t >0$. Choose $\lambda \ge  {\gamma\sigma }\sqrt{2\log(2(n-r_S))/n+ 2t/n} $. For the analysis estimator it holds that, $  \forall f \in \R^n$, with probability at least $1-e^{-x}-e^{-t}$,
\begin{equation*}
 \norm{\hat{f}-f^0}^2_n + 2 \lambda \norm{D_S \hat{f}}_1  \le  \norm{f-f^0}^2_n + \frac{\sigma^2}{n} \left(\sqrt{2x}+\sqrt{r_S} \right)^2+ 4 \lambda \norm{D f}_1.\\
\end{equation*}
\end{theorem}

\begin{proof}[Proof of Theorem \ref{analysis-t02s03}]
See Appendix \ref{analysis-appC}.
\end{proof}

\begin{remark}
Theorem \ref{analysis-t02s03} does not need the assumption that the (weighted) compatibility constant is bounded away from zero.
\end{remark}

\section{Oracle inequalities for the square root  analysis estimator}\label{analysis-s04}

In this section we study the square root analysis estimator, defined as
$$ \hat{f}_{\sqrt{}}:= \arg \min_{f \in \R^n} \left\{ \norm{Y-f}_n+ \lambda_0 \norm{Df}_1 \right\},\ \lambda_0>0.$$

Throughout this section we will make use of the following assumption.
\begin{assumption}\label{analysis-a01s02}
Assume for some $a>0$ that $n>8a$ and that for some $R>0, \eta\in (0,1)$

$$ \lambda_0\ge \frac{1}{1-\eta} R \text{ and } \norm{Df^0}_1\le c \sigma \sqrt{1-\sqrt{8a/n}}/ \lambda_0,$$
where,
$$c = \sqrt{\left(\frac{\eta}{2}- \frac{\sqrt{r_S}+\sqrt{2a}}{\sqrt{n-\sqrt{8an}}} \right)^2+4}-2.$$
We assume that $S \in \cals$ is s.t.
$${\eta}> 2 \frac{\sqrt{r_S}+\sqrt{2a}}{\sqrt{n-\sqrt{8an}}}.$$
\end{assumption}

\begin{note}
 Assumption \ref{analysis-a01s02} is also an assumption on $S$ and will thus  be a criterion to determine for which $S \in \cals$ our (oracle) results hold.
\end{note}

For the square root analysis estimator, to get the KKT conditions we have to make sure that $\he:= Y-\hat{f}\not=0$, i.e. that the estimator does not overfit.

The following lemma, showing that $\norm{\he}_n>0$, is an adaptation of Lemma 3.1 in \cite{vand16} to the case of the square root analysis estimator where the increments of the empirical process are bounded by the projection arguments found in \cite{dala17}.

\begin{lemma}\label{analysis-l03s02}
Let $S \in \mathcal{S}$ be an arbitrary active set satisfying Assumption \ref{analysis-a01s02} and let $a >0$. Choose $R \ge  \gamma \sqrt{\frac{2 \log(2(n-r_S))+2t}{n-1}},\ t \in (0, {(n-1)}/{2}-\log(2(n-r_S)))$.   Under Assumption \ref{analysis-a01s02} we have that, with probability at least $1-3e^{-a}-e^{-t}$,
$$ \left|\frac{\neh_n}{\nee_n}-1\right| \le \eta.$$
\end{lemma}

\begin{proof}[Proof of Lemma \ref{analysis-l03s02}]
See Appendix \ref{analysis-appD}.
\end{proof}

\begin{remark}
While Lemma 3.1 by \cite{vand16} only requires a  lower bound on $\nee_n$, Lemma \ref{analysis-l03s02} presented here requires that $\norm{\Pi_{\N(D_{-S})}\e}_n$ is upper and lower bounded and that $\norm{A_{\N(D_{-S})}\e}_n$  is lower bounded. It is the price to pay for a more refined technique to handle the increments of the empirical process.
\end{remark}

\begin{Coro}\label{analysis-c01s02}
Let $S \in \mathcal{S}$ be an arbitrary active set satisfying Assumption \ref{analysis-a01s02} and let $a >0$. Choose $R \ge  \gamma \sqrt{\frac{2 \log(2(n-r_S))+2t}{n-1}},\ t \in (0, {(n-1)}/{2}-\log(2(n-r_S)))$. Under Assumption \ref{analysis-a01s02}, we have that, with probability at least $1-3e^{-a}-e^{-t}$, $ \lambda_0 \neh_n \ge R \nee_n$.
\end{Coro}

\begin{proof}[Proof of Corollary \ref{analysis-c01s02}]
Under Assumption \ref{analysis-a01s02} on $\A\cap \RR$ Lemma \ref{analysis-l03s02} holds and thus $\neh_n \ge (1-\eta) \nee_n$.
It follows that $ \frac{1}{1-\eta}\ge \frac{\nee_n}{\neh_n}$.
By inserting this inequality into the assumption $ \lambda_0 \ge \frac{1}{1-\eta} R$ we get the claim.
\end{proof}



We now expose  oracle inequalities for the square root analysis estimator with fast and slow rates. The results are similar to Theorems \ref{analysis-t01s03} and \ref{analysis-t02s03} up to the constants and the assumptions one has to make.

\subsection{Fast rates with compatibility conditions}


\begin{theorem}[Oracle inequality with fast rates for the square root analysis estimator]\label{analysis-t01s04}
Let $S \in \mathcal{S}$ be an arbitrary active set satisfying Assumption \ref{analysis-a01s02} and let $a >0$. For $\eta \in (0,1)$, choose $\lambda_0 \ge\frac{1}{1-\eta}\gamma \sqrt{\frac{2 \log(2(n-r_S))+2t}{n-1}},\ t \in (0, {(n-1)}/{2}-\log(2(n-r_S)))$. Under  Assumption \ref{analysis-a01s02},  $\forall f \in \R^n$, it holds that, with probability at least $1-4e^{-a}-e^{-t}$,

\begin{equation*}
\norm{\hat{f}_{\sqrt{}}-f^0}^2_n  \le  \norm{f-f^0}^2_n + 16\sigma \lambda_0 \norm{D_{-S}f}_1 + \sigma^2\left(    \sqrt{\frac{2a}{n}}+  \sqrt{\frac{r_S}{n}}+  \frac{4 \lambda_0  \sqrt{r_S}}{\kappa(S,W)}\right)^2.
\end{equation*}
\end{theorem}

\begin{proof}[Proof of Theorem \ref{analysis-t01s04}]
See Appendix \ref{analysis-appD}.
\end{proof}

\subsection{Slow rates without compatibility conditions}

\begin{theorem}[Oracle inequality with slow rates for the square root analysis estimator]\label{analysis-t02s04}
Let $S \in \mathcal{S}$ be an arbitrary active set satisfying Assumption \ref{analysis-a01s02} and let $a >0$. For $\eta \in (0,1)$, choose $\lambda_0 \ge\frac{1}{1-\eta}\gamma \sqrt{\frac{2 \log(2(n-r_S))+2t}{n-1}},\ t \in (0, {(n-1)}/{2}-\log(2(n-r_S)))$. Under Assumption \ref{analysis-a01s02}, $\forall f \in \R^n$,  it holds that, with probability at least $1-4e^{-a}-e^{-t}$,
\begin{align*}
&\norm{\hat{f}_{\sqrt{}}-f^0}^2_n + 2(1-\eta)\sqrt{1-\sqrt{\frac{8a}{n}}}  \sigma \lambda_0  \norm{D_S \hat{f}_{\sqrt{}}}_1\\
&\le \norm{f-f^0}^2_n + \frac{\sigma^2}{n} \left(\sqrt{2a}+\sqrt{r_S} \right)^2 + 16 \sigma \lambda_0 \norm{D f}_1.\end{align*}
\end{theorem}

\begin{proof}[Proof of Theorem \ref{analysis-t02s04}]
See Appendix \ref{analysis-appD}.
\end{proof}

\begin{remark}
The claim of Theorem \ref{analysis-t02s04} implies also the simpler inequality
\begin{equation*}
 \norm{\hat{f}_{\sqrt{}}-f^0}^2_n  \le  \norm{f-f^0}^2_n + \frac{\sigma^2}{n} \left(\sqrt{2a}+\sqrt{r_S} \right)^2 + 16 \sigma \lambda_0 \norm{D f}_1.
\end{equation*}
\end{remark}

\begin{remark}
We can simplify for the ease of exposition  Assumption \ref{analysis-a01s02} on $\norm{Df^0}_1$ to $\norm{Df^0}_1= \bigo(1/\lambda_0)$. Note that if we take $\lambda_0\asymp \gamma \sqrt{{\log n}/{n}}$, then the assumption becomes
$$ \norm{Df^0}_1= \bigo\left( \sqrt{\frac{ n}{\gamma \log n}}\right).$$
If $\norm{Df^0}_1$ is growing with $n$, then the rates obtained with the slow rate oracle inequality by setting $f=f^0$ will be slower as well. In particular, if $\norm{Df^0}_1 \asymp 1/\lambda_0$, then Theorem \ref{analysis-t02s04} does not guarantee the convergence in $\norm{\cdot}_n$.
\end{remark}

\begin{remark}
The choice of the tuning parameter $\lambda_0$ depends on $S$ through $\gamma$. Therefore, in practice, the oracle inequalities will only hold for certain active sets $S$. To find out for which $S$ the oracle inequality holds with high probability we proceed as follows.

We choose $a>0$, $t \in (0, {(n-1)}/{2}-\log(2n))$, $\eta \in (0,1)$  and $\lambda_0>0$. Then, an active set $S$ for which the oracle inequality holds has to satisfy the following requirements:
$$ r_S < \left(\frac{\eta \sqrt{n-\sqrt{8an}}}{2}-\sqrt{2a} \right)^2,$$
and
$$ \gamma(S) \le \lambda_0 (1-\eta) \sqrt{\frac{n-1}{2 \log (2n)+2t}}.$$
\end{remark}

\section{Total variation}\label{analysis-s05}


\subsection{Incidence matrices}
Let $\vec{G}=(V,E)$ be a general directed graph, where the set $V=[n]$ is the set of vertices and the set $E=\{e_1, \ldots, e_m\}$ is the set of edges. Let $D_{\vec{G}} \in \{-1,0,1\}^{m \times n}$ be the incidence matrix of $\vec{G}$ (for more details see Subsubsection \ref{analysis-sss112}).
In this section we will set $D=D_{\vec{G}}$.
It is known that the rank of $D$ is given by the number of vertices of $\vec{G}$ minus its number of connected components. 

We now consider a set $S \in \cals$. Let us define the set of edges $E_S:=\{e_i\in E, i \in S\}$. The number of connected components of $\vec{G}_{-S}:=(V, E \setminus E_S)$ is $r_S$. These connected components can be any sort of graph: tree or non-tree graphs.

Let $n_1, \ldots, n_{r_S}$ be the number of vertices of each connected component $\vec{C}_i:= ([n_i],E_i ), i \in [r_S]$ of $\vec{G}_{-S}$. Let us define $n_{\min}:= \min \{n_1, \ldots, n_{r_S}\}$ and $n_{\max}:= \max \{n_1, \ldots, n_{r_S}\}$. The matrix $D_{-S}$ can be rewritten as block matrix by rearranging rows and columns. From now on, when writing $D_{-S}$ we intend the matrix in its block form.

By Lemma 1 in \cite{ijir65} we have that
$$ D^+_{-S}=\begin{pmatrix} D_{\vec{C}_i}& & \\ & \ddots & \\ & & D_{\vec{C}_{r_S}} \end{pmatrix}^+=\begin{pmatrix} D^+_{\vec{C}_i}& & \\ & \ddots & \\ & & D^+_{\vec{C}_{r_S}} \end{pmatrix}\in \R^{n\times (n-r_S)}.$$

\begin{remark}\label{analysis-r01s05}
The restriction to the class $\cals$ can be seen as a requirement to have an active set $S$ which makes sense. The incidence matrix of all connected graphs is of row rank $n-1$. However, graphs containing cycles, as the cycle graph or the two dimensional grid graph, have more than $n-1$ rows. The dimension $r_S$ of $\N_{D_{-S}}$ is the number of connected patches of the graph on which the signal is constant, if the active set is $S$. A non-empty active set, means that the signal should have at least two constant pieces, otherwise no edge would be active.

If the active set is $S=\emptyset$, then the dimension of $\N_{D_{-S}}$ is one. Now consider for instance the cycle graph. If $S=\{i\}$ for an $i \in [n]$, then the dimension of $\N_{D_{-S}}$ is still one. Thus this active set does not make sense at all since it would imply that we have a constant signal on the cycle graph but yet also a non-empty active set. Indeed, it is impossible to find a constant signal on a graph which results in some active edges. 

For tree graphs, we have that $\cals= \mathcal{P}([m])$, while for  graph structures containing cycles we have that $\cals\subset \mathcal{P}([m])$. In particular, for the cycle graph it holds that $\cals= \mathcal{P}([m])\setminus \{\{1\}, \ldots, \{n\}  \}$.
\end{remark}


\subsubsection{Trees and cycles}
If $\vec{G}$ is a tree or a cycle graph, then the connected components $\vec{C}_i=([n_i], E_i), i \in [r_S]$ of $\vec{G}_{-S}, S \not=\emptyset$ are tree graphs, i.e. connected graphs with $\abs{E_i}=n_i-1, i \in [r_S]$.  Let $D_{\vec{C}_i}\in \R^{(n_i-1)\times n_i},\ i \in [r_S]$ be the incidence matrices of the tree graphs $\vec{C}_i,\ i \in [r_S]$.

\begin{lemma}[Upper bound for the normalized inverse scaling factor]\label{analysis-l01s05}
Let $\vec{G}$ be a tree graph. Then,  $  \forall S\in \cals(D_{\vec{G}})$, the normalized inverse scaling factor $\gamma= \max_{i \in -S}\omega_i $ is bounded by
$$ \gamma \le \sqrt{\frac{n_{\max} +1}{4n}}.$$

Let $S\not=\emptyset$ and let $\vec{G}$ be a cycle graph. Then
$$ \gamma \le \sqrt{\frac{n_{\max} +1}{4n}}.$$
\end{lemma}

\begin{proof}[Proof of Lemma \ref{analysis-l01s05}]
See Appendix \ref{analysis-appE}.
\end{proof}

\subsubsection{Two dimensional grid graph}

We report and slightly adapt the bound on the normalized inverse scaling factor for the two dimensional grid graph by \cite{hutt16}.

\begin{lemma}[Proposition 4 in \cite{hutt16}]\label{analysis-l01s05-bis}
Let $\vec{G}$ be a two dimensional $\sqrt{n}\times \sqrt{n}$ grid graph.
Let $S \in \mathcal{S}$ be s.t. the connected components of $\vec{G}_{-S}$ are square two dimensional grid graphs.
Then, for some sufficiently large constant $C>0$, the normalized inverse scaling factor $\gamma= \max_{i \in -S}\omega_i $ is bounded by
$$ \gamma \le C \sqrt{\frac{\log(n_{\max})}{n}}.$$
\end{lemma}

%

\subsection{Fast rates}\label{analysis-fastrates}

To prove oracle inequalities with fast rates we need to find an explicit lower bound for the weighted compatibility constant.

Results for the analysis estimator on the path graph have already been obtained by \cite{orte18}. We extend them to the square root analysis estimator. Moreover, we also show that the tools developed in \cite{orte18} together with the new framework presented here, allow to handle the case of the cycle graph. We are aware of results treating the ${k}^{\text{th}}$ power graphs of cycles (\cite{hutt16}) but not of any oracle inequality implying the convergence of the mean squared error for the case of the cycle graph.

\subsubsection{Path graph}\label{analysis-ssspath}

We now consider the path graph $\vec{G}=([n], \{(1,2), \ldots, (n-1,n)  \}  )$, for which  $\cals=\mathcal{P}([m])$ and  $S=S$.

We see that $D_{-S}$ is a block matrix, where the blocks are incidence matrices of some smaller path graphs.  By recycling the proof of Lemma \ref{analysis-l01s05} we obtain that

$$ \norm{(D_{\vec{C}_i}^+)_j}_2= \sqrt{\frac{j (n_i-j)}{n_i}}.$$

The following lemma by \cite{vand18}, later also used in \cite{orte18}, allows us to lower bound $\kappa(S,W)$, for a diagonal matrix $W=\text{diag}(\{w_j\}_{j \in [n-1]})\in \R^{(n-1)\times(n-1)}$ with $\norm{W}_{\infty}\le 1$ and where by convention we choose $w_n=1$.

\begin{lemma}[Theorem 15 and Lemma 21 in \cite{vand18}]\label{analysis-l02s05}
Assume that $S\subseteq [m]$ is s.t. $n_1, n_{r_S}\ge 2$ and $n_i\ge 4, \forall i \in \{n_2, \ldots, n_{r_S-1}\}$. 
Then
$$ \frac{\sqrt{r_S}}{\kappa(S, \text{I}_{n-1})}\le \sqrt{nK},\text{ where } K= \frac{1}{n_1} + \sum_{i=2}^{r_S-1} \left(\frac{1}{\lfloor n_i/2\rfloor} + \frac{1}{\lceil n_i/2\rceil} \right) + \frac{1}{n_{r_S}}$$
and the inequality is tight. Moreover
$$ \frac{\sqrt{r_S}}{\kappa(S,W)}  \le  \frac{\sqrt{r_S}}{\kappa(S, \text{I}_{n-1})} + \sqrt{n \sum_{i=2}^n(w_i-w_{i-1})^2}.$$
\end{lemma}

\begin{proof}[Proof of Lemma \ref{analysis-l02s05}]
The first statement follows form Theorem 15 and the second from Lemma 21 in \cite{vand18}. The proofs are also exposed also in \cite{orte18}, in Lemmas 5.3-5.
\end{proof}

In \cite{orte18} it is explained that to bound the weak weighted compatibility constant for the path graph one needs to cut it into $s$ smaller modules. These modules lie around an edge in $S$ and consist of at least one additional edge on each side of the edge in $S$, see Figure \ref{analysis-f01}. Therefore we see that the assumption $n_i\ge 4, \forall i \in [s+1]$ guarantees that we are in a situation where the bounds on the weak weighted compatibility constant apply. Indeed, if $n_i \ge 4, \forall i \in [s+1]$ we have at least four vertices on the left and on the right of each edge in $S$ and thus we can decompose the path graph into modules being at least as large as the one shown in Figure \ref{analysis-f01}. Since the oracles inequalities with fast rates exposed here are based on the bound on the weighted weak compatibility constant by \cite{orte18}, for fast rates we will require that $n_i\ge 4, \forall i \in [s+1]$.


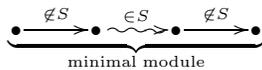
\begin{figure}[h]
$${\underbrace{\xymatrix@1{
  {\bullet} \ar[r]^{\not \in S} & {\bullet} \ar@{~>}[r]^{\in S} & {\bullet} \ar[r]^{\not \in S} & {\bullet}}}_{\text{minimal module}}}$$
  \caption{Illustration of the minimal module needed to bound the compatibility constant.}\label{analysis-f01}
\end{figure}

The edges not in $S$ between modules can be ignored. Each module needs at least $4$ vertices, s.t. we need $\abs{S}\le n/4$ to hope to be able to upper bound $\kappa^2(S,W)$ by using the method proposed by \cite{vand18}. Moreover, a vertex not involved in an edge in $S$ can only be involved in one module to obtain the bounds exposed in \cite{orte18}.

Note also that the weights in $w$ have a direct correspondence to the edges of the graph, where the edges in $S$ are s.t. $\omega_S=0$. Moreover, the weights for the edges between modules can be chosen arbitrarily when it comes to bounding $\norm{D_S f}_1- \norm{W_{-S} D_{-S} f}_1$ from above, even if a value for them can be  obtained by computation of the $\norm{\cdot}_n$-norm of the corresponding columns of $D^+_{-S}$.


We take the arbitrary decision to use the convention $w_n:=1$, as in Lemma \ref{analysis-l02s05}.

%
%

\begin{lemma}\label{analysis-l03s05}
Assume that $S\subseteq[ m]$ is s.t. $n_i\ge 4,\ \forall i\in [r_S]$.
We have that
$$\sum_{j=2}^n (w_i-w_{i-1})^2\le \frac{5}{\gamma^2}\frac{r_S}{n} \log \left(\frac{n}{r_S} \right).$$
\end{lemma}

\begin{proof}[Proof of Lemma \ref{analysis-l03s05}]
See the proof of Corollary 5.6 in \cite{orte18}.
\end{proof}

Let $D\in \R^{(n-1)\times n}$ be the incidence matrix of the path graph with $n$ vertices.
With the tools developed we can prove the following corollaries.\\

\underline{\textbf{Analysis estimator on the path graph}}\\

Corollary \ref{analysis-c01s05} below, is a result already found in \cite{orte18}. It is reported here for comparison with the analogous result obtained for the square root analysis estimator on the path graph (s. Corollary \ref{analysis-c03s05}). Corollary  \ref{analysis-c02s05} follows directly from Corollary \ref{analysis-c01s05}.

\begin{Coro}\label{analysis-c01s05}
Let $S\subseteq[m]$ be an arbitrary active set with s.t. $n_{\min}\ge 4$ and let $x, \ t >0$. Choose $\lambda\ge \sigma \sqrt{n_{\max} (\log(2n)+t)}/{n}$.
Then, $\forall f \in \R^n$, it holds that, with probability at least $1-e^{-t}-e^{-x}$,
\begin{eqnarray*}
\norm{\hat{f}-f^0}^2_n &\le& \norm{f-f^0}^2_n + 4 \lambda \norm{(Df)_{-S}}_1+ \sigma^2 \left(  \sqrt{\frac{2x}{n}}+  \sqrt{\frac{r_S}{n}} \right. \\
&&\left. +   \sqrt{ \frac{n_{\max} K}{n}(\log(2n)+t)} +    \sqrt{10\frac{r_S}{n} (\log(2n)+t) \log \left(\frac{n}{r_S} \right)} \right)^2.
\end{eqnarray*}
\end{Coro}

\begin{proof}[Proof of Corollary \ref{analysis-c01s05}]
See Appendix \ref{analysis-appE}.
\end{proof}

Due to the use  of the bound given by Lemma \ref{analysis-l03s05}, Corollary \ref{analysis-c01s05} assumes a minimal length condition. This condition does not depend on $n$ and is therefore weaker than the one found in \cite{gunt17}. Note that the choice of the tuning parameter depends both on $\sigma$ and $n_{\max}= n_{\max}(S)$.

The next corollary makes a stronger assumption on $S$.

\begin{Coro}\label{analysis-c02s05}
Let $S\subseteq[m]$ be an arbitrary active set with s.t. $n_{\min}=n_{\max}\ge 4$ and $n_{\max}$  even. Let $x, \ t >0$.
Choose $\lambda\ge {\sigma} \sqrt{\frac{\log(2n)+t}{r_Sn}}$.
Then, $\forall f \in \R^n$, it holds that,  with probability at least $1-e^{-t}-e^{-x}$,
\begin{eqnarray*}
\norm{\hat{f}-f^0}^2_n &\le& \norm{f-f^0}^2_n + 4 \lambda \norm{(Df)_{-S}}_1 + \sigma^2 \left(  \sqrt{\frac{2x}{n}}+  \sqrt{\frac{r_S}{n}}\right. \\
&& \left. +   \sqrt{ \frac{4r_S}{n}(\log(2n)+t)} +    \sqrt{10\frac{r_S}{n} (\log(2n)+t) \log \left(\frac{n}{r_S} \right)} \right)^2.
\end{eqnarray*}
\end{Coro}

\begin{proof}[Proof of Corollary \ref{analysis-c02s05}]
If $n_{\min}= n_{\max}$, then $n_i=n/r_S,\ \forall i \in [r_S]$. Moreover, 
$ K \le 4 r_S^2/n$ and the statement of Corollary \ref{analysis-c02s05} follows by plugging in these insights into Corollary \ref{analysis-c01s05}.
\end{proof}


Corollary \ref{analysis-c02s05} says that, if $n_{\min}= n_{\max}$, then  we can choose $\lambda$ smaller than the universal choice $\lambda \asymp \sigma \sqrt{\log n / n}$. The choice of the  constant-friendly tuning parameter in the two corollaries above assumes however the knowledge of some aspects of the oracle signal $f$ minimizing the right hand side and can be seen as a motivation to choose the tuning parameter smaller than the universal choice if we know or suspect a certain specific structure for it.
 These insights were already developed by \cite{dala17} and applied to total variation on the path graph in the case of slow rates.\\

\underline{\textbf{Square root analysis estimator on the path graph}}\\

We now extend the results obtained for the analysis estimator to the case of the square root analyisis estimator.

\begin{Coro}\label{analysis-c03s05}
Let $S\subseteq[m]$ be an arbitrary active set having $n_{\min}\ge 4$ and satisfying Assumption \ref{analysis-a01s02}. Let $a>0$ and $t \in (0, (n-1)/2 - \log(2(n-r_S)))$.
Choose $\lambda_0\ge \frac{1}{1-\eta} \sqrt{n_{\max} \frac{\log(2n)+t}{n(n-1)}}$.
Then, under Assumption \ref{analysis-a01s02}, $\forall f \in \R^n$,  for the square root version of the total variation regularized estimator over the path graph it holds that, with probability at least $1-e^{-t}-4e^{-a}$, 
\begin{align*}
&\norm{\hat{f}_{\sqrt{}}-f^0}^2_n \le \norm{f-f^0}^2_n + 16 \lambda_0 \sigma \norm{(Df)_{-S}}_1+ \sigma^2 \left(  \sqrt{\frac{2a}{n}}+  \sqrt{\frac{r_S}{n}} \right. \\
& \left. +   \frac{4}{1-\eta}\sqrt{ \frac{n_{\max} K}{n-1}(\log(2n)+t)} + \frac{4\sqrt{10}}{1-\eta}  \sqrt{\frac{r_S}{n-1} (\log(2n)+t) \log \left(\frac{n}{r_S} \right)} \right)^2.\\
\end{align*}
\end{Coro}

\begin{proof}[Proof of Corollary \ref{analysis-c03s05}]
The proof of Corollary \ref{analysis-c03s05} is analogous to the proof of Corollary \ref{analysis-c01s05}.
\end{proof}

\begin{Coro}\label{analysis-c04s05}
Let $S\subseteq[m]$ be an arbitrary active set having $n_{\min}= n_{\max}\ge 4$ with $n_{\max}$ even and satisfying Assumption \ref{analysis-a01s02}. Let $a>0$ and $t \in (0, (n-1)/2 - \log(2(n-r_S)))$.
Choose $\lambda_0\ge \frac{1}{1-\eta} \sqrt{n_{\max} \frac{\log(2n)+t}{n(n-1)}}$.
Then, under Assumption \ref{analysis-a01s02}, $\forall f \in \R^n$,  for the square root version of the total variation regularized estimator over the path graph it holds that, with probability at least $1-e^{-t}-4e^{-a}$, 
\begin{align*}
&\norm{\hat{f}_{\sqrt{}}-f^0}^2_n \le \norm{f-f^0}^2_n + 16 \lambda_0 \sigma \norm{(Df)_{-S}}_1 + \sigma^2 \left(  \sqrt{\frac{2a}{n}}+  \sqrt{\frac{r_S}{n}} \right.\\
& \left. +   \frac{4}{1-\eta}\sqrt{ \frac{4r_S}{n-1}(\log(2n)+t)} + \frac{4\sqrt{10}}{1-\eta}  \sqrt{\frac{r_S}{n-1} (\log(2n)+t) \log \left(\frac{n}{r_S} \right)} \right)^2.\\
\end{align*}
\end{Coro}

\begin{proof}[Proof of Corollary \ref{analysis-c04s05}]
If $n_{\min}= n_{\max}$, then $n_i=n/r_S,\ \forall i \in [r_S]$. Moreover, 
$ K \le 4 r_S^2/n$ and the statement of Corollary \ref{analysis-c04s05} follows by plugging in these insights into Corollary \ref{analysis-c03s05}.
\end{proof}

\begin{remark}
We notice that there is a tradeoff in the choice of $\eta$. A small $\eta$ will result in a narrower bound for $\neh_n$ in terms of $\nee_n$ and in smaller constants in the tuning parameter and in the oracle bound. However, it might result in a more restrictive condition on $S$ in Assumption \ref{analysis-a01s02}. 
\end{remark}

\subsubsection{Cycle graph}

We  consider the  the cycle graph $\vec{G}=([n], \{(1,2), \ldots, (n-1,n), (n,1)  \}  )$ and its incidence matrix $D\in \R^{n \times n}$. We have $\cals(D)= \mathcal{P}([m])\setminus \{\{1\}, \ldots, \{n\}  \}$.


We bound the weighted compatibility constant by cutting the graph into smaller modules as we explained in Subsubsection \ref{analysis-ssspath} for the path graph. 

By concatenating such modules, one can obtain a path graph. Whether or not the two ends of the path graph are joined by an edge is not relevant for the possibility to bound the compatibility constant and obtain an oracle inequality with fast rates, since the edges connecting such modules are neglected in the bound.

\begin{remark}
Note that for the path graph we have that $r_S=\abs{S}+1$, while for the cycle graph it holds that $r_S=\abs{S}$.
\end{remark}


\begin{Coro}\label{analysis-c07s05}
Assume that $S\in \cals$ is s.t. $n_i\ge 4,\ \forall i \in [r_S]  $. 
Then
$$ \frac{\sqrt{r_S}}{\kappa(S, \text{I}_{n})}\le \sqrt{nK'},\text{ where } K'= \sum_{i=1}^{r_S} \left(\frac{1}{\lfloor n_i/2\rfloor} + \frac{1}{\lceil n_i/2\rceil} \right) $$
and the inequality is tight. Moreover
$$ \frac{\sqrt{r_S}}{\kappa(S,W)}  \le  \frac{\sqrt{r_S}}{\kappa(S, \text{I}_{n})} + \sqrt{n \norm{Dw}_2^2}.$$
\end{Coro}

\begin{proof}[Proof of Corollary \ref{analysis-c07s05}]
Corollary \ref{analysis-c07s05} follows from Lemma \ref{analysis-l02s05} and from the considerations above.
\end{proof}

\begin{remark}
From Lemma \ref{analysis-l02s05} we get that, if $S \in \cals$ is s.t. $n_i \ge 4,\ \forall i \in [r_S]$, then
$$ \norm{Dw}^2_2 \le \frac{5}{\gamma^2} \frac{r_S}{n} \log \left(\frac{n}{r_S} \right).$$
\end{remark}

We now have all the tools to derive an oracle inequality for the total variation regularized estimator over the cycle graph and its square root version.\\

\underline{\textbf{Analysis estimator on the cycle graph}}\\
\begin{Coro}\label{analysis-c08s05}
 Let $S \in \cals\setminus \emptyset$ be an arbitrary active set with $n_{\min}\ge 4$ and let $x,t>0$. 
Choose $\lambda\ge \sigma \sqrt{n_{\max} (\log(2n)+t)}/{n}$.
Then, $\forall f \in \R^n$, for the total variation regularized estimator over the cycle graph  it holds that, with probability at least $1-e^{-t}-e^{-x}$, 
\begin{align*}
&\norm{\hat{f}-f^0}^2_n \le \norm{f-f^0}^2_n + 4 \lambda \norm{(Df)_{-S}}_1 +  \sigma^2 \left(  \sqrt{\frac{2x}{n}}+  \sqrt{\frac{r_S}{n}}\right. \\
& \left. +   \sqrt{ \frac{n_{\max} K'}{n}(\log(2n)+t)} +    \sqrt{10\frac{r_S}{n} (\log(2n)+t) \log \left(\frac{n}{r_S} \right)} \right)^2.
\end{align*}
\end{Coro}

An analogous version of Corollary \ref{analysis-c02s05} can be derived from Corollary \ref{analysis-c08s05}.\\

\underline{\textbf{Square root analysis estimator on the cycle graph}}\\

\begin{Coro}\label{analysis-c09s05}
Let $S \in \cals\setminus \emptyset$ be an arbitrary active set having $n_{\min}\ge 4$ and satisfying Assumption \ref{analysis-a01s02}. Let $a>0$ and $t \in (0, (n-1)/2 - \log(2(n-r_S)))$.
Choose $\lambda_0\ge \frac{1}{1-\eta} \sqrt{n_{\max} \frac{\log(2n)+t}{n(n-1)}}$.
Then, $\forall f \in \R^n$, for the square root version of the total variation regularized estimator over the cycle graph it holds that, with probability at least $1-e^{-t}-4e^{-a}$,
\begin{align*}
&\norm{\hat{f}_{\sqrt{}}-f^0}^2_n \le \norm{f-f^0}^2_n + 16 \lambda_0 \sigma \norm{(Df)_{-S}}_1+ \sigma^2 \left(  \sqrt{\frac{2a}{n}}+  \sqrt{\frac{r_S}{n}}\right.\\
& \left. +   \frac{4}{1-\eta}\sqrt{ \frac{n_{\max} K'}{n-1}(\log(2n)+t)} +  \frac{4\sqrt{10}}{1-\eta}  \sqrt{\frac{r_S}{n-1} (\log(2n)+t) \log \left(\frac{n}{r_S} \right)} \right)^2.\\
\end{align*}
\end{Coro}

An analogous version of Corollary \ref{analysis-c04s05} can be derived from Corollary \ref{analysis-c09s05}.

\subsection{Slow rates}\label{analysis-slowrates}

Note that in the case of the so-called slow rates we do not need to lower bound the compatibility constant.
\subsubsection{Trees and cycles}
In this subsection we identify the analysis operator $D$ with the incidence matrix of a general tree or cycle graph $\vec{G}$.\\

\underline{\textbf{Analysis estimator on trees and cycles}}\\

\begin{Coro}\label{analysis-c10s05}
Let $\vec{G}$ be a tree or a cycle graph. Let  $ S \in \cals(D_{\vec{G}})$ (and under the condition $S\not=\emptyset$ for cycle graphs) be arbitrary and let $x,t>0$. Choose $\lambda\ge {\sigma}\sqrt{n_{\max}(\log(2n)+t)}/n$. Then,  $\forall f \in \R^n$, we have that, with probability at least $1-e^{-x}-e^{-t}$,
\begin{equation*}
\norm{\hat{f}-f^0}^2_n  \le  \norm{f-f^0}^2_n + \frac{\sigma^2}{n} (\sqrt{2x}+ \sqrt{r_S})^2 + 4  \frac{\sigma}{n}\sqrt{n_{\max}(\log(2n)+t)}\norm{Df}_1.
\end{equation*}
\end{Coro}

\begin{proof}[Proof of Corollary \ref{analysis-c10s05}]
Corollary \ref{analysis-c10s05} follows by combining Theorem \ref{analysis-t02s03} and Lemma \ref{analysis-l01s05}.
\end{proof}

\begin{Coro}\label{analysis-c11s05}
Let $\vec{G}$ be a tree or a cycle graph. Let $ S \in \cals(D_{\vec{G}})$ (with the condition $S\not=\emptyset$ for cycle graphs) having $n_{\max}=n_{\min}$ be arbitrary and let $x,t>0$.
Choose $\lambda\ge{\sigma}\sqrt{\frac{\log(2n)+t}{r_S n}}$. Then, $\forall f \in \R^n$, we have that, with probability at least $1-e^{-x}-e^{-t}$, 
\begin{equation*}
\norm{\hat{f}-f^0}^2_n  \le  \norm{f-f^0}^2_n + \frac{\sigma^2}{n} (\sqrt{2x}+ \sqrt{r_S})^2 + 4{\sigma}\sqrt{\frac{\log(2n)+t}{r_S n}}\norm{Df}_1.
\end{equation*}
\end{Coro}

\underline{\textbf{Square root analysis estimator on trees and cycles}}\\

\begin{Coro}\label{analysis-c12s05}
Let $\vec{G}$ be a tree or a cycle graph.  Let $S\in \cals(D_{\vec{G}})$  (and under the condition $S\not=\emptyset$ for cycle graphs) be an arbitrary active set satisfying Assumption \ref{analysis-a01s02}. Let $a>0$ and $t \in (0, (n-1)/2 - \log(2(n-r_S)))$.
Choose $\lambda_0\ge \frac{1}{1-\eta} \sqrt{n_{\max} \frac{\log(2n)+t}{n(n-1)}}$.
Then, $\forall f \in \R^n$, it holds that under Assumption \ref{analysis-a01s02}, with probability at least $1-e^{-t}-4e^{-a}$, 
\begin{equation*}
 \norm{\hat{f}_{\sqrt{}}-f^0}^2_n  \le  \norm{f-f^0}^2_n + \frac{\sigma^2}{n} \left(\sqrt{2a}+\sqrt{r_S} \right)^2 + \frac{16 \sigma}{1-\eta} \sqrt{ \frac{n_{\max}(\log(2n)+t)}{n(n-1)}} \norm{D f}_1.
\end{equation*}
\end{Coro}

\begin{proof}[Proof of Corollary \ref{analysis-c12s05}]
Corollary \ref{analysis-c12s05} follows by combining Theorem \ref{analysis-t02s04} and Lemma \ref{analysis-l01s05}.
\end{proof}

\begin{Coro}\label{analysis-c13s05}
Let $\vec{G}$ be a tree or a cycle graph graph.   Let $S\in \cals(D_{\vec{G}})$  (and under the condition $S\not=\emptyset$ for cycle graphs) be an arbitrary active set having $n_{\max}=n_{\min}$ and satisfying Assumption \ref{analysis-a01s02}. Let $a>0$ and $t \in (0, (n-1)/2 - \log(2(n-r_S)))$.
Choose $\lambda_0\ge\frac{1}{1-\eta} \sqrt{ \frac{\log(2n)+t}{r_S(n-1)}}$.
Then, $\forall f \in \R^n$, it holds that under Assumption \ref{analysis-a01s02}, with probability at least $1-e^{-t}-4e^{-a}$, 
\begin{equation*}
 \norm{\hat{f}_{\sqrt{}}-f^0}^2_n  \le  \norm{f-f^0}^2_n + \frac{\sigma^2}{n} \left(\sqrt{2a}+\sqrt{r_S} \right)^2 + \frac{16 \sigma}{1-\eta} \sqrt{ \frac{\log(2n)+t}{r_S(n-1)}} \norm{D f}_1.\\
\end{equation*}
\end{Coro}

\subsubsection{Two dimensional grid graph}
In this subsection we identify the analysis operator $D$ with the incidence matrix of a square two dimensional grid graph $\vec{G}$.\\

\underline{\textbf{Analysis estimator on the two dimensional grid}}\\

\begin{Coro}\label{analysis-c14s05}
Let $\vec{G}$ be a square two dimensional grid graph. Let $ S \in \cals(D_{\vec{G}})$ be an arbitrary active set s.t. the connected components of $\vec{G}_{-S}$ are square two dimensional grid graphs and let  $x,t>0$. For a constant $C>0$ large enough, choose $\lambda\ge C {\sigma}\sqrt{\log n(\log(2n)+t)}/n$. Then,  $\forall f \in \R^n$, we have that, with probability at least $1-e^{-x}-e^{-t}$,
\begin{equation*}
\norm{\hat{f}-f^0}^2_n  \le  \norm{f-f^0}^2_n + \frac{\sigma^2}{n} (\sqrt{2x}+ \sqrt{r_S})^2 + C  \frac{\sigma}{n}\sqrt{\log n(\log(2n)+t)}\norm{Df}_1.
\end{equation*}
\end{Coro}

\begin{proof}[Proof of Corollary \ref{analysis-c14s05}]
Corollary \ref{analysis-c14s05} follows by combining Theorem \ref{analysis-t02s03} and Lemma \ref{analysis-l01s05-bis}.
\end{proof}

\underline{\textbf{Square root analysis estimator on the two dimensional grid}}\\

\begin{Coro}\label{analysis-c15s05}
Let $\vec{G}$ be a tree or a cycle graph. Let $ S\in \cals(D_{\vec{G}})$ be an arbitrary active set being s.t. the connected components of $\vec{G}_{-S}$ are square two dimensional grid graphs and satisfying Assumption \ref{analysis-a01s02}. Let $a>0$ and $t \in (0, (n-1)/2 - \log(2(n-r_S)))$.
For a constant $C>0$ large enough, choose $\lambda_0\ge \frac{C}{1-\eta} \sqrt{ \frac{\log n (\log(2n)+t)}{n(n-1)}}$.
Then, $\forall f \in \R^n$, it holds that under Assumption \ref{analysis-a01s02}, with probability at least $1-e^{-t}-4e^{-a}$, 
\begin{equation*}
 \norm{\hat{f}_{\sqrt{}}-f^0}^2_n  \le  \norm{f-f^0}^2_n + \frac{\sigma^2}{n} \left(\sqrt{2a}+\sqrt{r_S} \right)^2 + \frac{C \sigma}{1-\eta} \sqrt{ \frac{\log n (\log(2n)+t)}{n(n-1)}} \norm{D f}_1.
\end{equation*}
\end{Coro}

\begin{proof}[Proof of Corollary \ref{analysis-c15s05}]
Corollary \ref{analysis-c15s05} follows by combining Theorem \ref{analysis-t02s04} and Lemma \ref{analysis-l01s05-bis}.
\end{proof}

\subsubsection{Comparison with other results}

Consider Corollary \ref{analysis-c11s05} with the choice $f=f^0$ and assume that $\sigma$ does not depend on $n$. Then the following holds with probability at least $1-e^{-x}-e^{-t}$.
\begin{itemize}
\item With $r_S\asymp n^{1/3} (\log(2n)+t)^{1/3}\norm{Df^0}_1^{2/3}$, then $ \norm{\hat{f}-f^0}^2_n= \bigo(n^{-2/3} (\log(2n)+t)^{1/3} \norm{Df^0}_1^{2/3})$ and $\lambda$ explicitely depends on $f^0$.
\item With $r_S \asymp n^{1/3} (\log(2n)+t)^{1/3}$, then $ \norm{\hat{f}-f^0}^2_n= \bigo(n^{-2/3} (\log(2n)+t)^{1/3} \norm{Df^0}_1)$ and $\lambda$ does not explicitely depend on $f^0$.
\end{itemize}

One can reason analogously starting from Corollary \ref{analysis-c13s05} for the square root analysis estimator.

%

In both  cases, if $\norm{Df^0}= \bigo(1)$ we obtain that $ \norm{\hat{f}-f^0}^2_n = \bigo(n^{-2/3}\log^{1/3}(n))$.
However, it is known that the minimax rate for that case (when the graph considered is the path graph) is $ \norm{\hat{f}-f^0}^2_n= \bigo(n^{-2/3})$
and thus our results lead to a redundant log-term. The result about the minimax rate over the class of functions with bounded total variation obtained by entropy calculations (\cite{mamm97-2} and references therein) are not constant-friendly, so that it may well be that, for  $n$ small enough, the log-term is  smaller than the constants of the entropy arguments.

The same remark applies to the case of tree graphs of bounded maximal degrees. For such graphs, \cite{padi17} proved that the minimax rate of estimation of $f^0: \norm{D_{\vec{G}}f^0}_1 \le C$ is $n^{-2/3} C^{2/3}$. Moreover, they proved by entropy arguments that the total variation regularized estimator achieves the minimax rate. We prove that this minimax rate is achieved by the (square root) total variation regularized estimator up to a log term by using constant-friendly arguments (cf. Corollary \ref{analysis-c11s05} and \ref{analysis-c13s05}).

We thus saw that for the path graph, the constant-friendly projection argument introduced by \cite{dala17} to handle the increments of the empirical process might produce optimal rates up to a log-term  for both the  total variation regularized estimator and the square root total variation regularized estimator.

Another question is whether we can retrieve almost minimax rates by Corollary \ref{analysis-c14s05} for  $D_{\vec{G}}$ being the incidence matrix of a two dimensional grid graph. For that case, the minimax rate is $\sqrt{{\log n}/{n}}$ \cite{sadh16} and an oracle inequality proved by \cite{hutt16} almost retrieves it. Moreover, a natural scaling for that case is $ \norm{D_{\vec{G}}f^0}_1\asymp n^{1/2}$ (\cite{sadh16}). Note that the part of Assumption \ref{analysis-a01s02} concerning $\norm{D_{\vec{G}}f^0}_1$, which translates to
$ \norm{Df^0}_1= \bigo\left( n / \log n \right)$, is thus satisfied.

Thus, for $t,\ x>0$ fixed, from Corollaries \ref{analysis-c14s05} and \ref{analysis-c15s05} we get that, if $S_0$ is s.t. the connected components of $\vec{G}_{-S_0}$ are square two dimensional grid graphs and
$$r_{S_0} = \bigo(n^{-1/2}\log n),$$
 under the canonical scaling $\norm{Df^0}_1\asymp n^{1/2}  $ we have the rate
$$ \log n/\sqrt{n},$$
which corresponds to the minimax rate up to a log term. Note however that, due to the utilization of Lemma \ref{analysis-l01s05-bis}, Corollaries \ref{analysis-c14s05} and \ref{analysis-c15s05}, from which this insight is derived, are not constant-friendly.
\section{Conclusion}
We  introduced a class of active sets dependent on the analysis operator $D$, to which it is natural to restrict the attention. Indeed, as some examples from total variation regularization on graphs show, there can be some elements of $\mathcal{P}([m])$ which can not be seen as true active sets of any signal, depending on the graph structure.

We then derived oracle inequalities with fast rates under some compatibility conditions and oracle inequalities with slow rates. The results with fast rates show that, if one can find a suitable bound on the weighted weak compatibility constant, the analysis estimator and its square root version are adaptive, i.e. they can adapt to the unknown sparsity of $Df^0$. For both the analysis and the square root analysis estimators, the results with slow rates were used as tool to retrieve in a simple and constant-friendly way minimax rates obtained by entropy calculations, at the price of an extra log factor. The  choice of the tuning parameters $\lambda$ and $\lambda_0$, which includes some information about the structure of the analysis operator $D$ and of the active set $S$ via the inverse scaling factor $\gamma$, seems to be advantageous in theoretical terms and allows us to show that the ``slow'' rates can almost match the minimax lower bound for the total variation regularized estimator on graph structures as the path graph and tree graphs with bounded maximal degree.

We obtained parallel and very similar results for both the analysis and the square root analysis estimators. The differences in these results come from the fact that for the square root analysis estimator we first have to prove that the estimator does not overfit and that the KKT conditions hold.  In spite of being mathematically more involved, the results for the square root analysis estimator tell us that we can get with high probability  theoretical guarantees being very similar to the ones obtained for the analysis estimator by choosing a tuning parameter not depending on the unknown noise level. This fact might be helpful in practice and might speak in favor of the utilization of the square root analysis estimator.

We then narrowed down our results to (square root) total variation regularized estimators over graphs. For fast rates we considered the cases of the path graph and of the cycle graph. In these cases we were able to show that the compatibility conditions are satisfied.

For the case of slow rates, we obtained  oracle inequalities matching up to a log term the optimal rate over the path graph, the two dimensional grid graph and tree graphs of maximal bounded degree. These results do not require any compatibility condition. 

These oracle inequalities can be interpreted in two senses. Either we can choose a smaller tuning parameter depending on $S$ and obtain better rates. Or we can choose a larger tuning parameter not depending on $S$ and get worse rates. This might be a justification for incorporating eventual prior knowledge of  $S$ into  the tuning parameter.

The main tool used to derive the oracle inequalities presented in this paper is a bound on the increments of the empirical process inspired by the projection arguments by \cite{dala17}. This bound is very simple and constant-friendly, while entropy bounds are more involved and can have large constants. There are two routes one can take after having bounded the increments of the empirical process by projection arguments. Either one uses a more refined version of the bound on the increments of the empirical process and then bounds the compatibility constant to derive fast rates. Or one bounds the increments of the empirical process in a rougher way and obtains oracle inequalities with slow rates. In this last case one only needs to bound the inverse scaling factor. Bounds on the inverse scaling factor can be very simple and constant-friendly, while bounds on the compatibility constant can sometimes lead to large constants (cf. \cite{vand19}). Moreover, results with slow rates have been shown to almost retrieve the minimax rate in a constant-friendly way also in other settings, for instance in higher order total variation regularization (\cite{vand19}). 
If we compare the results obtained by entropy calculations with our results with slow rates, we see that, at the expense of a log term, we are able to retrieve almost the same rate by two simple steps: the constant-friendly bound on the increments of the empirical process and the bound on the inverse scaling factor. The bound on the inverse scaling factor is constant-friendly for graph structures as tree graphs and cycle graphs, while the bound on the inverse scaling factor for the two dimensional grid graph we borrow from \cite{hutt16} is more involved. For total variation regularized estimators on the path graph and on tree graphs of bounded maximal degree, we thus obtain nonasymptotic counterparts, in form of oracle inequalities with slow  rates, to results found in the previous literature (\cite{mamm97-2, padi17}).

A question for further investigation is the possibility to use the framework exposed here to obtain oracle inequalities with fast rates for other graph structures. The answer depends on the ability to lower bound the compatibility constant for graphs other than tree graphs and cycles. We leave this questions to future research.


\appendix

\section{Probability inequalities}\label{analysis-appA}

We expose three lemmas helping us to deal with the random part of the oracle inequalities.

\begin{lemma}[The maximum of $p$ random variables, Lemma 17.5 in \cite{vand16}]\label{l175vand16}
Let $V_1, \ldots, V_{p}$ be real valued random variables. Assume  $\forall j \in \{1,\ldots, p\}$ and $\forall r>0$ that $\ex[e^{r \abs{V_j}} ]\leq 2 e^{\frac{r^2}{2}}$.
Then, $\forall t>0$
\begin{equation*}
\pr\left(\max_{1\leq j \leq p} \abs{V_j}\geq \sqrt{2\log (2p)+2t} \right) \leq e^{-t}.
\end{equation*}
\end{lemma}

\begin{lemma}[The special case of $\chi^2$ random variables, Lemma 1 in \cite{laur00}, Lemma 8.6 in \cite{vand16}]\label{l86vand16}
Let $X\sim\chi^2_d$. Then, $\forall x>0$
\begin{equation*}
\pr\left(X\geq d+2\sqrt{dx}+2x \right)\leq e^{-x} \text{ and }  \pr \left(X \le d-2\sqrt{dx} \right)\leq e^{-x}.
\end{equation*}
\end{lemma}

\begin{remark}
Note that from Lemma \ref{l86vand16} it follows that
\begin{equation*}
\pr\left(\sqrt{X}\leq \sqrt{d}+\sqrt{2x} \right)\geq \pr\left(X\leq d+2\sqrt{dx}+2x \right)\geq 1-e^{-x}.
\end{equation*}
\end{remark}

\begin{lemma}[Lemma 8.1 in \cite{vand16}]\label{prob3}
For $n\ge 2$, let $\epsilon \sim \N_n(0, \sigma^2\text{I}_n)$. Then, $\forall u \in \R^n: \norm{u}_n=1$ we have that, for $t\in (0,(n-1)/2)$,
$$ \pr \left(\frac{u'\epsilon}{n\nee_n}> \sqrt{\frac{2t}{n-1}} \right)\le 2e^{-t}.$$
\end{lemma}

\begin{remark}
Let $u_1, \ldots, u_p\in \R^n$ be vectors. Then by the union bound and by Lemma \ref{prob3} we have that for $t'\in (0,(n-1)/2)$
$$ \pr\left(\max_{i \in [p]} \frac{\abs{u'_i\e}}{n\norm{u_i}_n \nee_n} > \sqrt{\frac{2t'}{n-1}} \right)\le 2p e^{-t'}.$$
Now select $t=t'-\log(2p)$. Then we have that for $t\in (0,(n-1)/2-\log(2p))$,
$$ \pr\left(\max_{i \in [p]} \frac{\abs{u'_i\e}}{n\norm{u_i}_n \nee_n} > \sqrt{\frac{2 \log(2p)+2t}{n-1}} \right)\le e^{-t}.$$
\end{remark}
\section{Proofs of Section \ref{analysis-s03}}\label{analysis-appC}

\subsection{Basic inequality}
The case of the analysis estimator is more simple than the one of the square root analysis estimator, because we have the basic inequality without assuming any extra conditions.

\begin{lemma}[Basic inequality]\label{analysis-l01s02}
For the analysis estimator we have the so called basic inequality, i.e. $\forall f \in \R^n$
\begin{equation*}
\norm{\hat{f}-f^0}^2_n+\norm{\hat{f}-f}^2_n  \le \norm{f-f^0}^2_n + \frac{2\epsilon'(\hat{f}-f)}{n} + 2\lambda ( \norm{Df}_1- \norm{D\hat{f}}_1).
\end{equation*}
\end{lemma}

\begin{proof}[Proof of Lemma \ref{analysis-l01s02}]
The KKT conditions for the analysis estimator write as
$$ \frac{Y-\hat{f}}{n}= \lambda D' \partial \norm{D\hat{f}}_1.$$
Thanks to the chain rule of the subdifferential, $D'\partial \norm{D\hat{f}}_1$ is the subdifferential of $\norm{Df}_1$ with respect to $f$ at $\hat{f}$. We have that, for $\hat{f}\in \R^n$, $ {\hat{f}'(Y-\hat{f})}/{n}= \lambda \norm{D\hat{f}}_1$ and that, for a generic $f \in \R^n$, $ {f'(Y-\hat{f})}/{n}= \lambda (Df)' \partial \norm{D\hat{f}}_1 \le \lambda \norm{Df}_1$, where the last inequality follows by the dual norm inequality and by the fact that $ \norm{\partial \norm{D\hat{f}}_1}_{\infty} \le 1$.

By subtracting the first of the two above expressions from the second, we find that

$$ \frac{(f-\hat{f})'(f^0-\hat{f})}{n} \le \frac{\epsilon'(\hat{f}-f)}{n} + \lambda ( \norm{Df}_1- \norm{D\hat{f}}_1).$$

By polarization we obtain the basic inequality
\begin{equation*}
\norm{\hat{f}-f^0}^2_n+\norm{\hat{f}-f}^2_n  \le  \norm{f-f^0}^2_n + \frac{2\epsilon'(\hat{f}-f)}{n} + 2\lambda ( \norm{Df}_1- \norm{D\hat{f}}_1).
\end{equation*}
\end{proof}



\subsection{Bound on the increments of the empirical process}

\begin{lemma}\label{analysis-l02s02}
Let $S \in \cals$ be arbitrary and $x, \ t>0$. Choose $\lambda \ge {\gamma\sigma }\sqrt{2\log(2(n-r_S))/n+ 2t/n}$. Then, $\forall f \in \R^n$, it holds that, with probability at least $1-e^{-x}-e^{-t}$,
\begin{align*}
\frac{\epsilon'f}{n}&\le \frac{\lambda}{\gamma} \norm{\Omega_{-S}D_{-S}f}_1 +\sigma \left(\sqrt{\frac{2x}{n}} + \sqrt{\frac{r_S}{n}} \right)\norm{f}_n\\
&\le \lambda  \norm{D_{-S}f}_1 + \sigma \left(\sqrt{\frac{2x}{n}} + \sqrt{\frac{r_S}{n}} \right)\norm{f}_n.
\end{align*}
\end{lemma}

\begin{proof}[Proof of Lemma \ref{analysis-l02s02}]
We have that
$$ \frac{\epsilon'f}{n}= \underbrace{\frac{\epsilon'\Pi_{\N^{\perp}(D_{-S})}f}{n}}_{1.} + \underbrace{\frac{\epsilon'\Pi_{\N(D_{-S})}f}{n}}_{2.}.$$

\begin{enumerate}
\item

We have that, since $D_{-S}$ is of full rank,
$$
\frac{\epsilon'\Pi_{\N^{\perp}(D_{-S})}f}{n} =\frac{\epsilon'D_{-S}^{+} D_{-S}f}{n}.
$$

For $\lambda>0$ define the set
\begin{eqnarray*} \mathcal{T}&:=&\left\{\left|\frac{\epsilon' d_i^+}{n}\right|\le \frac{\lambda \norm{d_i^+}_n}{\gamma}, \forall i \in [m-s] \right\}\\
&=&\left\{\max_{i\in[m-s]}\left|\frac{\epsilon'd^+_i}{\sigma\norm{d^+_i}_2}\right|\leq \frac{\lambda \sqrt{n}}{\gamma\sigma}   \right\}= \left\{\max_{1\leq i\leq m-s}\abs{V_i}\leq \frac{\lambda \sqrt{n}}{\gamma\sigma}\right\},
\end{eqnarray*}
where $V_i={\epsilon'd^+_i}/{(\sigma\norm{d^+_i}_2)} \sim\N(0,1), i \in [m-s]$, since $\e'd^+_i\sim\N(0,\sigma^2 \norm{d^+_i}^2_2)$.

Since $\gamma= \norm{\Omega}_{\infty}$, on $\mathcal{T}$ we have that
$$
\frac{\epsilon'D_{-S}^{+} D_{-S}f}{n} \le \frac{\lambda}{\gamma} \norm{\Omega_{-S}D_{-S}f}_1 \le  \lambda  \norm{D_{-S}f}_1.
$$

To find a lower bound on $ \pr(\mathcal{T})$ we apply Lemma \ref{l175vand16} to $\mathcal{T}$.

The moment generating function of $\abs{V_i}$ is $
\ex\left[e^{r\abs{V_i}} \right]=2(1-\Phi(-r))e^{\frac{r^2}{2}}\leq 2 e^{\frac{r^2}{2}},\ \forall r>0$.

Choosing, for some $t>0$, $\lambda\ge{\gamma\sigma }\sqrt{2(\log(2(n-r_S))+ t)/n}$, e.g. $\lambda={\gamma\sigma }\sqrt{2(\log(2n)+ t)/n}$, and applying Lemma \ref{l175vand16} with $p=m-s=n-r_S$ and $t>0$, we obtain that $\pr(\mathcal{T})\geq 1-e^{-t}$.

\item
We have that
$$
\frac{\epsilon'\Pi_{\N(D_{-S})}f}{n}
\le \norm{\Pi_{\N(D_{-S})}\epsilon}_n \norm{f}_n.
$$

For $x>0$, define the set 
$$ \mathcal{X}:=\left\{ \norm{\Pi_{\N(D_{-S})}\epsilon}_n \leq \sqrt{\frac{\sigma^2}{n}}\left(\sqrt{r_S}+\sqrt{2x}\right) \right\}.$$

On $\mathcal{X}$ we have that $$\norm{\Pi_{\N(D_{-S})}\epsilon}_n \leq \sqrt{\frac{\sigma^2}{n}}\left(\sqrt{r_S}+\sqrt{2x}\right).$$

Since $\N(D_{-S})$ is a linear space of dimension $r_S$, we have that $$ \frac{\norm{\Pi_{\N(D_{-S})}\epsilon}^2_2}{{\sigma^2}}\sim \chi^2_{r_S}.$$
Moreover note that
\begin{equation*}
 \norm{\Pi_{\N(D_{-S})}\epsilon}^2_n=\frac{\sigma^2}{n}{\frac{\norm{\Pi_{\N(D_{-S})}\epsilon}^2_2}{{\sigma^2}}}.
\end{equation*}
By applying Lemma \ref{l86vand16} for some $x>0$ we thus get that $\pr(\mathcal{X})\geq 1-e^{-x}$.

\end{enumerate}
\end{proof}

\begin{remark}
To obtain fast rates by using compatibility conditions one makes use of the more refined bound given by Lemma \ref{analysis-l02s02} involving $\norm{\Omega_{-S} D_{-S} f}_1/\gamma$. This term will flow into the weighted compatibility constant.

To obtain slow rates without needing compatibility conditions one utilizes the less refined version of the bound given by Lemma \ref{analysis-l02s02} involving $ \norm{ D_{-S} f}_1$.
\end{remark}

\subsection{Proof of the oracle inequalities}

\begin{proof}[Proof of Theorem \ref{analysis-t01s03}]

%

By Lemma \ref{analysis-l01s02} we have the basic inequality.
By the triangle inequality, we have 
\begin{eqnarray*}
\norm{Df}_1- \norm{D\hat{f}}_1 &=& \norm{D_S f}_1- \norm{D_S \hat{f}}_1 - (\norm{D_{-S}f}_1+ \norm{D_{-S}\hat{f}}_1) +  2 \norm{D_{-S}f}_1\\
&\le & \norm{D_S (\hat{f}-f)}_1 - \norm{D_{-S}(\hat{f}-f)}_1 + 2 \norm{D_{-S}f}_1.
\end{eqnarray*}

%
%
%

We now handle the random part, which is constituted by an increment of the empirical process, by using Lemma \ref{analysis-l02s02}.
By Lemma \ref{analysis-l02s02} we have that 
with probability at least $1-e^{-x}-e^{-t}$,
$$
\frac{\epsilon'(\hat{f}-f)}{n} \le \frac{\lambda}{\gamma} \norm{\Omega_{-S}D_{-S}(\hat{f}-f)}_1 + \left(\sqrt{\frac{2x}{n}} + \sqrt{\frac{r_S}{n}} \right)\norm{\hat{f}-f}_n
.$$
Putting the pieces together, we get that, 
\begin{eqnarray*}
\norm{\hat{f}-f^0}^2_n + \norm{\hat{f}-f}^2_n & \le & \norm{f-f^0}^2_n + 4 \lambda \norm{D_{-S}f}_1\\
&&+ 2 \norm{\hat{f}-f}_n \left(\sqrt{\frac{2x}{n}} + \sqrt{\frac{r_S}{n}} \right)\\
&&+ 2 \lambda (\norm{D_S(\hat{f}-f)}_1- \norm{D_{-S}(\hat{f}-f)}_1)\\
&&+ 2 \lambda \norm{\Omega_{-S} D_{-S}(\hat{f}-f)}_1/\gamma\\
&=& \norm{f-f^0}^2_n + 4 \lambda \norm{D_{-S}f}_1\\
&&+ 2 \norm{\hat{f}-f}_n \left(\sqrt{\frac{2x}{n}} + \sqrt{\frac{r_S}{n}} \right)\\
&&+ 2 \lambda (\norm{D_S(\hat{f}-f)}_1- \norm{W_{-S}D_{-S}(\hat{f}-f)}_1).
\end{eqnarray*}
If $\kappa(S,W)>0$ we have that
$$ \norm{D_S(\hat{f}-f)}_1- \norm{W_{-S}D_{-S}(\hat{f}-f)}_1 \le \frac{\sqrt{r_S}\norm{\hat{f}-f}_n}{\kappa(S,W)}$$
and thus
\begin{eqnarray*}
\norm{\hat{f}-f^0}^2_n + \norm{\hat{f}-f}^2_n& \le & \norm{f-f^0}^2_n + 4 \lambda \norm{D_{-S}f}_1\\
&&+ 2 \norm{\hat{f}-f}_n \left(\sqrt{\frac{2x}{n}} + \sqrt{\frac{r_S}{n}} + \frac{\lambda\sqrt{r_S}}{\kappa(S,W)} \right)\\
 & \le & \norm{f-f^0}^2_n + 4 \lambda \norm{D_{-S}f}_1\\
 &&+ \norm{\hat{f}-f}^2_n +   \left(\sqrt{\frac{2x}{n}} + \sqrt{\frac{r_S}{n}} + \frac{\lambda\sqrt{r_S}}{\kappa(S,W)} \right)^2,
\end{eqnarray*}
 where the last inequality follows by $2ab \le a^2+ b^2, a,b \in \R$.
 
 The term $\norm{\hat{f}-f}^2_n$ cancels out and we get the statement of the theorem.
\end{proof}

\begin{proof}[Proof of Theorem \ref{analysis-t02s03}]
By Lemma \ref{analysis-l01s02} we have the basic inequality.
%
By Lemma \ref{analysis-l02s02}, 
 we have that with probability at least $1-e^{-x}-e^{-t}$,
\begin{eqnarray*}
\frac{\epsilon'(\hat{f}-f)}{n} & \le & \lambda \norm{D_{-S}(\hat{f}-f)}_1 + \sigma \left(\sqrt{\frac{2x}{n}}+\sqrt{\frac{r_S}{n}} \right)\norm{\hat{f}-f}_n\\
& \le & \lambda  \norm{D_{-S}(\hat{f}-f)}_1 + \frac{1}{2}\sigma^2 \left(\sqrt{\frac{2x}{n}}+\sqrt{\frac{r_S}{n}} \right)^2+\frac{1}{2}\norm{\hat{f}-f}^2_n.\\
\end{eqnarray*}
We thus get that
\begin{equation*}
\norm{\hat{f}-f^0}^2_n  \le  \norm{f-f^0}^2_n + \frac{\sigma^2}{n} (\sqrt{2x}+\sqrt{r_S} )^2 + 2 \lambda (\norm{D_S f}_1 + \norm{D_{-S}f}_1- \norm{D_S \hat{f}}_1 ).
\end{equation*}
\end{proof}
\section{Proofs of Section \ref{analysis-s04}}\label{analysis-appD}

Define for $S \in \cals$
$$\hat{R}:= \max_{i \in -S} \frac{\abs{\e'd_i^+}}{\norm{\e}_n \norm{d_i^+}_n n}.$$

For $a>0, R>0$,  define the sets $ \RR:= \left\{\gamma \hat{R} \le R \right\}$,
\begin{align*}
\A&:= \left\{\frac{\norm{\Pi_{\N(D_{-S})}\e}^2_2}{\sigma^2}-r_S\in \left[-2\sqrt{ar_S},  + 2 \sqrt{ar_S}+2a \right] \right\}\\
&\cup \left\{\frac{\norm{A_{\N(D_{-S})}\e}^2_2}{\sigma^2}\ge n-r_S - 2 \sqrt{a(n-r_S)} \right\}
\end{align*}
and 
$$ \A':= \A\cap \left\{\norm{A_{\N(D_{-S})}\e}^2_2/ \sigma^2 \le n-r_S + 2 \sqrt{a(n-r_S)} + 2a \right\}.$$

Note that on $\A'$ we have that, by the Cauchy-Schwarz inequality,
\begin{equation*}
\nee_2^2 = \norm{\Pi_{\N(D_{-S})}\e}^2_2 +\norm{A_{\N(D_{-S})}\e}^2_2 \le {\sigma^2} (n + \sqrt{8an}+ 4a)\le n \sigma^2 (1+ \sqrt{4a/n})^2.
\end{equation*}

\begin{remark}\label{analysis-r01s02}
By Lemma \ref{l86vand16} (Lemma 1 in \cite{laur00}) we have that for $a>0$ both $ \pr(\A)\ge 1-3e^{-a}$ and $ \pr(\A') \ge 1-4e^{-a}$ hold true.

Moreover by Lemma \ref{prob3}   (Lemma 8.1 in \cite{vand16}) and using the union bound, we see that if we choose
$$ R \ge \gamma \sqrt{\frac{2 \log(2(n-r_S))+2t}{n-1}},\ t \in (0, {(n-1)}/{2}-\log(2(n-r_S)))$$
we have that $ \pr(\RR)\ge 1-e^{-t}$.
Thus, by such a choice of $R$ we get that
$$ \pr(\A\cap \RR)\ge 1-3e^{-a}-e^{-t} \text{ and } \pr(\A'\cap \RR)\ge 1-4e^{-a}-e^{-t}.$$
\end{remark}

\begin{remark}
 Motivated by a more simple exposition of the results, we chose the same parameter $a$ for the upper and lower bounds for both $\norm{\Pi_{\N(D_{-S})}\e}_n$ and $\norm{A_{\N(D_{-S})}\e}_n$. However one could of course choose four different parameters, say $a_i,\ i \in [4]$, for the four different bounds and obtain results holding with probability $1-e^{-t}-\sum_{i=1}^3 e^{a_i}$ resp. $1-e^{-t}-\sum_{i=1}^4 e^{a_i}$.
 \end{remark}

\subsection{Proving that the square root analysis estimator does not overfit}

\begin{proof}[Proof of Lemma \ref{analysis-l03s02}]
Assumption \ref{analysis-a01s02} expresses a particular choice of the constant $c$ in Proposition \ref{l01s11} below.
For $\eta\in (0,1)$ we have that $\eta/2 \le \eta /(1+ \eta)$ and thus the choice of $c$ in Assumption \ref{analysis-a01s02} satisfies the upper bound given by Proposition \ref{l01s11} (see below), which then holds, since all of its assumtpions are satisfied and we consider the sets $\A \cap \RR$.

The choice of $c$ implies that $q=\eta/2$ and that $c\le \eta /2$. Thus the claim follows.

 By Remark \ref{analysis-r01s02}, if we choose $a >0$ and $R \ge\gamma \sqrt{\frac{2 \log(2(n-r_S))+2t}{n-1}},\ t \in (0, {(n-1)}/{2}-\log(2(n-r_S)))$, then $\pr(\A \cap \R) \ge 1-3e^{-a}-e^{-t}$.

\end{proof}

\begin{proposition}[The square root analysis estimator does not overfit]\label{l01s11}

Assume for some $a>0$ that $n>8a$ and that for some $R>0, \eta\in (0,1)$

$$ \lambda_0\ge \frac{1}{1-\eta} R \text{ and } \norm{Df^0}_1\le c \sigma \sqrt{1-\sqrt{8a/n}}/ \lambda_0,$$
where 
$$c < \sqrt{\left(\frac{\eta}{1+\eta}- \frac{\sqrt{r_S}+\sqrt{2a}}{\sqrt{n-\sqrt{8an}}} \right)^2+4}-2.$$
We assume that $S \in \mathcal{S}$ is s.t. 
$$ \frac{\eta}{1+\eta}> \frac{\sqrt{r_S}+\sqrt{2a}}{\sqrt{n-\sqrt{8an}}}.$$

Let $$q:= 2 \frac{\sqrt{r_S}+\sqrt{2a}}{\sqrt{n-\sqrt{8an}}}+ (c+2)^2-4.$$

Let $a >0$. Choose $R\ge\gamma \sqrt{\frac{2 \log(2(n-r_S))+2t}{n-1}},\ t \in (0, {(n-1)}/{2}-\log(2(n-r_S)))$. Then with probability at least $1-3e^{-a}-e^{-t}$ it holds that
$$ (1+c)\norm{\e}_n \ge \norm{\he}_n \ge (1- \eta q /(\eta-q))\norm{\e}_n.$$
\end{proposition}

\begin{proof}[Proof of Proposition \ref{l01s11}, based on the proof of Lemma 3.1 by \cite{vand16}]

On the set $\A$ we have that, the Cauchy-Schwarz inequality,
$$ \nee_2^2/\sigma^2 \ge n-2\sqrt{a}(\sqrt{r_S}+ \sqrt{n-r_S}) \ge n-\sqrt{8an}.$$
Thus,
$$ \nee_n\ge \sigma \sqrt{1-\sqrt{8a/n}} \text{ and by assumption we have that } \norm{Df^0}_1 \le c\nee_n /\lambda_0.$$

We now show an upper and a lower bound for $\neh_n$. 

\medskip
\textbf{\underline{Upper bound:}}\\

Since the estimator $\hf_{\sqrt{}}$ minimizes the objective function we have that
$$ \norm{Y-\hf_{\sqrt{}}}_n + \lambda_0 \norm{D\hf_{\sqrt{}}}_1 \le \norm{Y-f^0}_n + \lambda_0 \norm{Df^0}_1.$$

It follows that
$$\neh_n \le \nee_n + \lambda_0 \norm{Df^0}_1 \le (1+c) \nee_n.$$

\medskip
\textbf{\underline{Lower bound:}}\\

Note that, by the triangle inequality, we have that
$$ \neh_n = \norm{\e-(\hf_{\sqrt{}}-f^0)}_n \ge \nee_n - \norm{\hf_{\sqrt{}}-f^0}_n.$$

Thus the lemma follows if we can prove a bound of the type $ \norm{\hf_{\sqrt{}}-f^0}_n \le \text{const.}\nee_n$, with leading constant in $(0,1)$. We are not allowed to use the KKT conditions. Instead we use the convexity of the loss function and of the penalty.

Define for $t \in (0,1)$ the convex combination $ \hft:= t \hf_{\sqrt{}} + (1-t) f^0$ and its residuals

$$ \het:=Y-\hft= \e-(\hft-f^0)=t\he+(1-t)\e.$$

Choose
$$ t= \frac{\eta \nee_n}{\eta \nee_n+ \norm{\hf_{\sqrt{}}-f^0}_n}.$$

Then
$$ \norm{\hft-f^0}_n= t \norm{\hf_{\sqrt{}}-f^0}_n=\frac{\eta \nee_n \norm{\hf_{\sqrt{}}-f^0}_n}{\eta \nee_n+ \norm{\hf_{\sqrt{}}-f^0}_n}\le \eta \nee_n.$$

We thus get that
$$ \net_n \ge \nee_n - \norm{\hft-f^0}_n \ge (1-\eta) \nee_n.$$

By the convexity of the loss and of the penalty and by the fact that $\hf_{\sqrt{}}$ is a minimizer of the objective function it follows that
\begin{align*}
\net_n+ \lambda_0 \norm{D\hft}_1  &\le  t(\neh_n+ \lambda_0 \norm{D\hf_{\sqrt{}}}_1) + (1-t)(\nee_n + \lambda_0 \norm{Df^0}_1)\\
&\le \nee_n + \lambda_0 \norm{Df^0}_1.
\end{align*}

By squaring the inequality we get that
\begin{equation*}
\net_n^2+ 2 \lambda_0 \net_n \norm{D\hft}_1 + \lambda_0^2 \norm{D\hft}_1^2  \le \nee^2_n + 2 \lambda_0 \nee_n \norm{Df^0}_1 + \lambda_0^2 \norm{Df^0}_1^2.
\end{equation*}
We have that
\begin{equation*}
\net_n^2 = \norm{\e-(\hft-f^0)}^2_n = \nee^2_n - \frac{2\e'(\hft-f^0)}{n} + \norm{\hft-f^0}^2_n.
\end{equation*}

By combining the squared inequality with the lower bound for $\net_n$ and the expression for $\net_n^2$ we get that
\begin{equation*}
\norm{\hft-f^0}^2_n \le 2 \lambda_0 \nee_n \norm{Df^0}_1- 2 \lambda_0 (1-\eta)\nee_n \norm{D\hft}_1  + \lambda_0^2 \norm{Df^0}_1^2  + \frac{2\e'(\hft-f^0)}{n}.
\end{equation*}
On $\RR$, for an $S$ satisfying the assumptions of the lemma, we have that
\begin{eqnarray*}
\frac{\e'(\hft-f^0)}{n} & \le & \frac{\rho}{\gamma} R \nee_n (\norm{D_{-S}f}_1+ \norm{D_{-S}f^0}_1) + \norm{\Pi_{\N(D_{-S})}\e}_n \norm{\hft-f^0}_n\\
& \le &  R \nee_n (\norm{Df}_1+ \norm{Df^0}_1) + \norm{\Pi_{\N(D_{-S})}\e}_n \norm{\hft-f^0}_n.\\
\end{eqnarray*}

Thus
\begin{eqnarray*}
 &&\norm{\hft-f^0}^2_n - 2 \norm{\Pi_{\N(D_{-S})}\e}_n \norm{\hft-f^0}_n\\
&\le & 2 (\lambda_0+R) \nee_n \norm{Df^0}_1 - 2 (\lambda(1-\eta)-R)\nee_n \norm{D\hft}_1 +  \lambda_0^2 \norm{Df^0}_1^2\\
& \le & 4 \lambda_0 \nee_n \norm{Df^0}_1 + \lambda_0^2 \norm{Df^0}_1^2 = \nee_n^2 \left((\lambda_0\norm{Df^0}_1/ \nee_n + 2)^2-4 \right)\\
&\le & \nee_n^2 \underbrace{\left((c+2)^2-4 \right)}_{=: c'}.
\end{eqnarray*}

Moreover we have that
\begin{align*}
\norm{\hft-f^0}^2_n - 2 \norm{\Pi_{\N(D_{-S})}\e}_n\norm{\hft-f^0}_n = &(\norm{\hft-f^0}_n -  \norm{\Pi_{\N(D_{-S})}\e}_n)^2\\
&-\norm{\Pi_{\N(D_{-S})}\e}_n^2.
\end{align*}
Thus we obtain that
$$\left(\norm{\hft-f^0}_n - \norm{\Pi_{\N(D_{-S})}\e}_n\right)^2 \le\norm{\Pi_{\N(D_{-S})}\e}_n^2 + c' \nee^2_n$$
and
\begin{align*}
\norm{\hft-f^0}_n  &\le \norm{\Pi_{\N(D_{-S})}\e}_n + \sqrt{\norm{\Pi_{\N(D_{-S})}\e}_n^2 + c' \nee^2_n}\\
&\le 2\norm{\Pi_{\N(D_{-S})}\e}_n + \sqrt{c'} \nee_n.
\end{align*}

Note that
$$ \nee^2_2 = \norm{\Pi_{\N(D_{-S})}\e}^2_2 + \norm{A_{\N(D_{-S})}\e}^2_2,$$
By using the spectral decomposition, $\Pi_{\N(D_{-S})}\R^{n \times n} $ can be written as $PP'$, where $P \in \R^{n \times r_S}$ is s.t. $P'P=\text{I}_{r_S}$.
Moreover $A_{\N(D_{-S})}\in \R^{n \times n}$ can be written as $QQ'$, where $Q\in \R^{n \times (n-r_S)}$ is s.t. $Q'Q= \text{I}_{n-r_S}$ and $Q'P=0$.

Let $u:= P'\e\in \R^{r_S}$ and $v:=Q'\e\in \R^{n-r_S}$. We have that
$ u \sim \N_{r_S} (0, \sigma^2 \text{I}_{r_S})$,
$ v \sim \N_{n-r_S} (0, \sigma^2 \text{I}_{n-r_S})$ and $u$ and $v$ are independent.
We have that $\norm{\Pi_{\N(D_{-S})}\e}^2_2=\norm{u}^2_2$ and that $ \norm{A_{\N(D_{-S})}\e}^2_2=\norm{v}^2_2$.
It follows that
$$ \nee^2_2/\sigma^2= \underbrace{\norm{u}^2_2/\sigma^2}_{\sim \chi^2_{r_S}} + \underbrace{\norm{v}^2_2/\sigma^2}_{\sim \chi^2_{n-r_S}}$$
and thus the two terms are independent and can be handled separately.

On $\A$ we have that
$$\frac{\norm{\Pi_{\N(D_{-S})}\e}^2_2}{\nee^2_2}= \frac{\norm{u}^2_2}{\norm{u}^2_2+\norm{v}^2_2}\le \frac{(\sqrt{r_S}+ \sqrt{2a})^2}{n-\sqrt{8an}}.$$
Therefore
$$ \norm{\Pi_{\N(D_{-S})}\e}_n \le \frac{\sqrt{r_S}+ \sqrt{2a}}{\sqrt{n-\sqrt{8an}}} \nee_n =: p \nee_n \asymp \sqrt{\frac{r_S}{n}} \nee_n.$$
It follows that
$$ \norm{\hft-f^0}_n \le (2p+ \sqrt{c'})\nee_n =: q \nee_n .$$
By expressing $\norm{\hft-f^0}_n$  more explicitly we get that
$$ \frac{\eta\norm{\hf_{\sqrt{}}-f^0}_n}{\eta\nee_n +\norm{\hf_{\sqrt{}}-f^0}_n} \le q \text{ and thus } \norm{\hf_{\sqrt{}}-f^0}_n \le \frac{q \eta}{\eta- q} \nee_n.$$
We conclude that
$$ \neh_n \ge \left(1- \frac{\eta q}{\eta-q} \right)\nee_n.$$

The last step is to find out how to choose $c$ s.t. $q\eta / (\eta-q)<1$.
We get that $q< \eta/(1+\eta)$, hence
$$ c'< \left(\frac{\eta}{1+\eta}-p \right)^2 \text{ and thus } c< \sqrt{\left(\frac{\eta}{1+\eta}-p \right)^2+4}-2.$$
Note that we also get the assumption $p< \eta /(1+\eta)$, which results in the assumption
$$\frac{\eta}{1+\eta}> \frac{\sqrt{r_S}+\sqrt{2a}}{\sqrt{n-\sqrt{8an}}}. $$
Note that the result holds on $\A \cap \R$, which by Remark \ref{analysis-r01s02} has probability at least $1-3e^{-a}-e^{-t}$ for $a>0$ and $R \ge\gamma \sqrt{\frac{2 \log(2(n-r_S))+2t}{n-1}},\ t \in (0, {(n-1)}/{2}-\log(2(n-r_S)))$.
\end{proof}

\subsection{Basic inequality}
\begin{lemma}\label{analysis-l04s02}
Let $S \in \mathcal{S}$ be an arbitrary active set satisfying Assumption \ref{analysis-a01s02} and let $a >0$. For $\eta \in (0,1)$, choose $\lambda_0 \ge\frac{1}{1-\eta}\gamma \sqrt{\frac{2 \log(2(n-r_S))+2t}{n-1}},\ t \in (0, {(n-1)}/{2}-\log(2(n-r_S)))$. Under Assumption \ref{analysis-a01s02}, it holds that $\forall f \in \R^n$, with probability at least $1-3e^{-a}-e^{-t}$, 
\begin{equation*}
\norm{\hat{f}_{\sqrt{}}-f^0}^2_n+\norm{\hat{f}_{\sqrt{}}-f}^2_n  \le  \norm{f-f^0}^2_n + \frac{2\epsilon'(\hat{f}_{\sqrt{}}-f)}{n} + 2\lambda_0 \neh_n ( \norm{Df}_1- \norm{D\hat{f}_{\sqrt{}}}_1).
\end{equation*}
\end{lemma}

\begin{proof}[Proof of Lemma \ref{analysis-l04s02}]
Under Assumption \ref{analysis-a01s02}, on $\A \cap \RR$ the KKT conditions hold
$$ \frac{Y-\hf_{\sqrt{}}}{n}= \lambda_0 \neh_n D' \partial \norm{D\hf_{\sqrt{}}}_1.$$
We then obtain the basic inequality as in Lemma \ref{analysis-l01s02} (cf. also Lemma 2 in \cite{stuc17}).
Note that by Remark \ref{analysis-r01s02}, the choice of $\lambda_0$ implies that $\pr(\A \cap \R) \ge 1-3e^{-a}-e^{-t}$.
\end{proof}

\subsection{Bound on the increments of the empirical process}

\begin{lemma}\label{analysis-l05s02}
Let $S \in \mathcal{S}$ be an arbitrary active set satisfying Assumption \ref{analysis-a01s02} and let $a >0$. For $\eta \in (0,1)$, choose $\lambda_0 \ge\frac{1}{1-\eta}\gamma \sqrt{\frac{2 \log(2(n-r_S))+2t}{n-1}},\ t \in (0, {(n-1)}/{2}-\log(2(n-r_S)))$.
Under Assumption \ref{analysis-a01s02} we have that $\forall f \in \R^n$, with probability at least $1-3e^{-a}-e^{-t}$  
$$ \frac{\e'f}{n}\le \lambda_0 \neh_n \norm{\Omega_{-S}D_{-S}f}_1/\gamma + \sqrt{\frac{\sigma^2}{n}(r_S + 2 \sqrt{ar_S}+2a)} \norm{f}_n.$$
\end{lemma}

\begin{proof}[Proof of Lemma \ref{analysis-l05s02}]
On $\RR$, by using the decomposition in antiprojection and projection onto the nullspace of $D_{-S}$ and by applying the dual norm inequality to the second term we have that
\begin{align*}
\frac{\e'f}{n} &\le \frac{\e'D_{-S}^+ D_{-S}f}{n}+ \norm{\Pi_{\N(D_{-S})}\e}_n \norm{f}_n\\
&\le  R \nee_n \norm{\Omega_{-S}D_{-S}f}_1/\gamma +  \sqrt{\frac{\sigma^2}{n}(r_S + 2 \sqrt{ar_s}+2a)} \norm{f}_n.
\end{align*}
Moreover, on $\A \cap \RR$, under Assumption \ref{analysis-a01s02}, by Corollary \ref{analysis-c01s02} we have that $R\nee_n \le \lambda_0 \neh_n$ and thus the claim follows.
Note that the choice of $\lambda_0$ implies, by Remark \ref{analysis-r01s02}, that $\pr(\A \cap \R) \ge 1-3e^{-a}-e^{-t}$.
\end{proof}

\subsection{Proof of the oracle inequalities}

\begin{proof}[Proof of Theorem \ref{analysis-t01s04}]
We work under Assumption \ref{analysis-a01s02} on $\A'\cap \RR$.
By combining Lemma \ref{analysis-l04s02} and Lemma \ref{analysis-l05s02}, we get that, in complete analogy to the proof of Theorem \ref{analysis-t01s03},
\begin{equation*}
\norm{\hat{f}_{\sqrt{}}-f^0}^2_n \le \norm{f-f^0}^2_n + 4 \lambda_0 \neh_n \norm{D_{-S}f}_1 + \left(   \sigma \sqrt{\frac{2a}{n}}+ \sigma \sqrt{\frac{r_S}{n}}+  \frac{\lambda_0 \neh_n \sqrt{r_S}}{\kappa(S,W)}\right)^2.
\end{equation*}
Moreover, by Corollary \ref{analysis-c01s02}, we have that on $\A'$
$$ \neh_n \le (1+\eta)\nee_n \le (1+\eta) (1+\sqrt{4a/n})\sigma.$$
Thus we get that
\begin{eqnarray*}
\norm{\hat{f}_{\sqrt{}}-f^0}^2_n & \le & \norm{f-f^0}^2_n + 4(1+\eta) (1+\sqrt{4a/n})\sigma \lambda_0  \norm{D_{-S}f}_1\\
&+& \sigma^2\left(    \sqrt{\frac{2a}{n}}+  \sqrt{\frac{r_S}{n}}+  \frac{(1+\eta) (1+\sqrt{4a/n}) \lambda_0  \sqrt{r_S}}{\kappa(S,W)}\right)^2.
\end{eqnarray*}
Since Assumption \ref{analysis-a01s02} implies that $\eta<1$ and $n>8a$ we get that
$(1+\eta) (1+\sqrt{4a/n})\le 4$ and 
\begin{equation*}
\norm{\hat{f}_{\sqrt{}}-f^0}^2_n \le \norm{f-f^0}^2_n + 16\sigma \lambda_0 \norm{D_{-S}f}_1 + \sigma^2\left(    \sqrt{\frac{2a}{n}}+  \sqrt{\frac{r_S}{n}}+  \frac{4 \lambda_0  \sqrt{r_S}}{\kappa(S,W)}\right)^2.
\end{equation*}
By Remark \ref{analysis-r01s02} and the choice of $\lambda_0$ in the statement of the theorem, we have that $\pr(\A'\cap \R)\ge 1-4e^{-a}-e^{-t}$.
\end{proof}

\begin{proof}[Proof of Theorem \ref{analysis-t02s04}]
We work under Assumption \ref{analysis-a01s02} on $\A'\cap \RR$.
By Lemma \ref{analysis-l04s02} and Lemma \ref{analysis-l05s02} we get that, in analogy with the proof of Theorem \ref{analysis-t02s03},
\begin{equation*}
\norm{\hat{f}_{\sqrt{}}-f^0}^2_n + 2 \lambda_0 \neh_n \norm{D_S \hat{f}_{\sqrt{}}}_1 \le \norm{f-f^0}^2_n + \frac{\sigma^2}{n} \left(\sqrt{2x}+\sqrt{r_S} \right)^2+ 4 \lambda_0 \neh_n \norm{D f}_1.
\end{equation*}
By Corollary \ref{analysis-c01s02} we have that
$$2\nee_n \ge (1+\eta)\nee_n \ge \neh_n \ge (1-\eta) \nee_n.$$
Moreover on $\A'$
$$ 2\sigma\ge \sigma (1+ \sqrt{4a/n})\ge \nee_n \ge \sigma \sqrt{1-\sqrt{8a/n}}.$$
We thus get that
\begin{align*}
 &\norm{\hat{f}_{\sqrt{}}-f^0}^2_n + 2(1-\eta)\sqrt{1-\sqrt{\frac{8a}{n}}}  \sigma \lambda_0   \norm{D_S \hat{f}_{\sqrt{}}}_1 \\
 &\le  \norm{f-f^0}^2_n + \frac{\sigma^2}{n} \left(\sqrt{2a}+\sqrt{r_S} \right)^2 + 16 \sigma \lambda_0 \norm{D f}_1.
\end{align*}
By Remark \ref{analysis-r01s02} and the choice of $\lambda_0$ in the statement of the theorem, we have that $\pr(\A'\cap \R)\ge 1-4e^{-a}-e^{-t}$.
\end{proof}
\section{Proofs of Section \ref{analysis-s05}}\label{analysis-appE}

\begin{proof}[Proof of Lemma \ref{analysis-l01s05}]
Notice that for a cycle graph, all elements of $\mathcal{S}\setminus \emptyset$ have at least $s=2$ (cf. Remark \ref{analysis-r01s05}). Thus under the assumption $S \not=\emptyset$, bounding $\gamma$ for the cycle graph reduces to bounding $\gamma$ for a tree graph.

Let $D \in \R^{(n-1)\times n}$ be the incidence matrix of a directed tree graph rooted at vertex 1. Let $D^+\in \R^{n \times (n-1)}$ be its Moore-Penrose pseudoinverse. By Lemma 2.2 in \cite{orte19-1} we have that $D^+$ can be obtained as $ D^+ = ( \text{I}_n-\mathbb{I}_n/n)X_{-1}$, where $X= \begin{pmatrix} (1,0, \ldots, 0)\\ D \end{pmatrix}^{-1}$. As pointed out in \cite{orte18}, $X$ has the meaning of the rooted path matrix of the tree graph considered. Thus, the columns of $X_{-1}$ contain a minimum of $1$ and a maximum of $(n-1)$ entries having value 1, while the remaining entries are zeroes.

Let $i$ be the number of entries having value 1 of a column of $X_{-1}$. Let $v(i)\in \R^n,\ i \in [n]$ denote any vector with $i$ entries having value 1 and $(n-i)$ entries having value 0.
Define $g(i,n):= \norm{(\text{I}_n-\mathbb{I}_n/n)v(i)}^2_2$.
We have that $g(i,n)= i(1-i/n)^2+ (n-i)(i/n)^2= {i(n-i)}/{n},\ i \in [n-1]$.
The maximum of $g(i,n)$ for a given $n$ is reached at $i=n/2$ if $n$ is even and at $i\in \{\lfloor n/2 \rfloor, \lceil n/2 \rceil \}$ if $n$ is odd.

%
%
Moreover, $\max_{i \in [n-1]}g(i,n)$ is increasing in $n$ and $ g(i,n) \le \frac{n+1}{4}, \forall i \in [n-1],\ \forall n$.
Therefore, the $\ell^2$-norm of a column of $D^+_{-S}$ will never be greater than the greatest possible $\ell^2$-norm of a column of $D^+_{\vec{C}_i}$.
We thus have that
$$ \gamma= \max_{j \in [n-1]}\norm{d^+_j}_n\le \max_{i \in [n_{\max}-1]}\sqrt{\frac{g(i, n_{\max})}{n}}\le \sqrt{\frac{n_{\max} +1}{4n}}.$$
\end{proof}

\begin{proof}[Proof of Corollary \ref{analysis-c01s05}]
By Lemma \ref{analysis-l01s05} we have that
$$ \gamma\le \sqrt{\frac{n_{\max}+1}{4n}}, \text{ therefore we choose }  \lambda^2\ge  \sigma^2 n_{\max} \frac{\log(2n)+t}{n^2}.$$
By combining the above with Lemma \ref{analysis-l02s05}, Lemma \ref{analysis-l03s05} and Theorem \ref{analysis-t01s03} we get Corollary \ref{analysis-c01s05}.
\end{proof}

\newpage

\bibliographystyle{imsart-nameyear}

\bibliography{library}

\end{document}

%% file: analysis-pp-concordance.tex
\Sconcordance{concordance:analysis-pp.tex:analysis-pp.Rnw:%
1 116 1 1 0 48 1}
\Sconcordance{concordance:analysis-pp.tex:./analysis-Section1.Rnw:ofs 166:%
1 230 1}
\Sconcordance{concordance:analysis-pp.tex:analysis-pp.Rnw:ofs 397:%
167}
\Sconcordance{concordance:analysis-pp.tex:./analysis-Section3.Rnw:ofs 398:%
1 64 1}
\Sconcordance{concordance:analysis-pp.tex:./analysis-Section4.Rnw:ofs 463:%
1 123 1}
\Sconcordance{concordance:analysis-pp.tex:./analysis-Section5.Rnw:ofs 587:%
1 479 1}
\Sconcordance{concordance:analysis-pp.tex:./analysis-Section6.Rnw:ofs 1067:%
1 18 1}
\Sconcordance{concordance:analysis-pp.tex:analysis-pp.Rnw:ofs 1086:%
172 17 1}
\Sconcordance{concordance:analysis-pp.tex:./analysis-AppendixA.Rnw:ofs 1104:%
1 38 1}
\Sconcordance{concordance:analysis-pp.tex:analysis-pp.Rnw:ofs 1143:%
191}
\Sconcordance{concordance:analysis-pp.tex:./analysis-AppendixC.Rnw:ofs 1144:%
1 226 1}
\Sconcordance{concordance:analysis-pp.tex:./analysis-AppendixD.Rnw:ofs 1371:%
1 255 1}
\Sconcordance{concordance:analysis-pp.tex:./analysis-AppendixE.Rnw:ofs 1627:%
1 61 1}
\Sconcordance{concordance:analysis-pp.tex:analysis-pp.Rnw:ofs 1689:%
195 12 1}